\documentclass[preprint,12pt]{elsarticle}

\usepackage{subdepth}
\usepackage[percent]{overpic}
\usepackage{amsmath,amstext,amsbsy,amssymb,mathdots}
\usepackage{booktabs,bm}
\usepackage{mathrsfs}
\usepackage{tikz}
\usepackage{natbib}
\usepackage{todonotes}
\usepackage{courier}
\usetikzlibrary{positioning}
\usetikzlibrary{shapes,arrows}
\usepackage[fleqn,tbtags]{mathtools}
\usepackage{amsfonts}
\usepackage{relsize}
\usepackage{xcolor,colortbl}
\usepackage{psfrag}
\usepackage{alltt}
\usepackage{arydshln}
\usepackage{enumerate}
\usepackage{epstopdf}
\usepackage{enumitem}
\usepackage{multirow}
\usepackage{algorithmic}
\usepackage{framed}
\usepackage{lineno,hyperref}
\modulolinenumbers[5]
\journal{Journal of \LaTeX\ Templates}
\colorlet{lightgray}{gray!40}

\newtheorem{lemma}{Lemma}
\newtheorem{definition}{Definition}
\newtheorem{theorem}{Theorem}
\newproof{proof}{Proof}
\newtheorem{corollary}{Corollary}

\newcommand{\erank}{{\rm rank}_\epsilon}
\newcommand{ \rank}{{\rm rank}}
\newcommand{\rb}[1]{ r_{#1}(B)^{-1}}

\renewcommand{\vec}[1]{\mathbf{#1}}

\newcolumntype{"}{@{\hskip\tabcolsep\vrule width 1pt\hskip\tabcolsep}}
\setlist[description]{font=\normalfont\space}

\journal{Linear Algebra and its Applications}

\begin{document}

\begin{frontmatter}
\title{On the singular values of matrices with high displacement rank}

\author[cm]{Alex Townsend\fnref{fn1}}
\ead{townsend@cornell.edu}

\author[cam]{Heather Wilber\fnref{fn2}}
\ead{hdw27@cornell.edu}

\fntext[fn1]{This work is supported by National Science Foundation grant No.~1645445.}
\fntext[fn2]{This material is based upon work supported by the National Science Foundation Graduate Research Fellowship under grant No.~DGE-1650441.}

\address[cm]{Department of Mathematics, Cornell University, Ithaca, NY  14853}
\address[cam]{Center for Applied Mathematics, Cornell University, Ithaca, NY  14853}

\begin{abstract}
We introduce a new ADI-based low rank solver for $AX-XB=F$, where $F$ has rapidly decaying singular values. Our approach results in both theoretical and practical gains, including (1) the derivation of new bounds on singular values for classes of matrices with high displacement rank, (2) a practical algorithm for solving certain Lyapunov and Sylvester matrix equations with high rank right-hand sides, and (3) a collection of low rank Poisson solvers that achieve spectral accuracy and optimal computational complexity.      
\end{abstract}


\end{frontmatter}

\section{Introduction}\label{sec:Introduction}
Matrices with rapidly decaying singular values appear with extraordinary frequency in computational mathematics. Such matrices are said to have low numerical rank, and a collection of low rank approximation methods used in particle simulation~\cite{greengard1987fast}, reduced-order modeling~\cite{antoulas2005approximation, benner2005dimension} and matrix completion~\cite{candes2009exact} has developed around exploiting them. An explicit bound on the numerical rank of a matrix requires bounding its singular values, and this is generally difficult. However, bounds can be derived for families of matrices that have displacement structure~\cite{beckermann2016singular, penzl2000eigenvalue, sabino2007solution}. In this paper, we derive explicit bounds on the singular values of matrices with displacement structure in cases where known bounds fail to be informative.  Our method is constructive and leads to an efficient low rank approximation scheme that we call the factored-independent alternating direction implicit (FI-ADI) method.  It can be used, among other things, to develop fast and spectrally-accurate low rank solvers for certain elliptic partial differential equations (PDEs). 

The matrix \smash{$X \in \mathbb{C}^{m \times n}$} satisfying the Sylvester matrix equation
\begin{equation} 
\label{eq:Sylv}
AX-XB = F,  \quad A \in \mathbb{C}^{m \times m}, \quad B \in \mathbb{C}^{n \times n}, 
\end{equation} 
is said to possess displacement structure, with an $(A, B)$-displacement rank of $\rho = {\rm rank}(F)$. Explicit bounds on the singular values of structured matrices, such as Cauchy,  L{\"o}wner,  real  positive definite Hankel and real Vandermonde matrices, can be derived by observing that these matrices satisfy~\eqref{eq:Sylv} for some specific triple $(A, B, F)$, with $F$ of rank 1 or 2~\cite{beckermann2016singular}. 
Other work has focused primarily on the case where~\eqref{eq:Sylv} is a Lyapunov matrix equation (i.e., \smash{$B = -A^*$, $F = F^*$}, with \smash{$M^*$} denoting the Hermitian transpose of $M$) and $\rank(F) = 1$. Closed-form solutions, approximation by exponential sums, Cholesky factorizations, and the convergence properties of iterative methods have been used to derive  bounds on the singular values of $X$~\cite{antoulas2002decay, grasedyck2003existence, sabino2007solution, kressner2016low,  simoncini2016computational}. 

Rapidly decaying singular values imply that $X$ is well approximated by a low rank matrix.
\begin{definition} 
\label{def:erank}
Let  \smash{$X \in \mathbb{C}^{m \times n}$}, $m \geq n$, and $ 0 < \epsilon <  1$ be given. The \smash{$\epsilon$}-rank of $X$, denoted by \smash{$\erank(X)$},  is the smallest integer $k$ such that 
\[ \sigma_{k+1}(X) \leq \epsilon \; \|X\|_2,\]
where  \smash{$\sigma_{j}(X)$} denotes the $j$th singular value of $X$, \smash{$\sigma_1(X) = \|X\|_2$,} and \smash{$\sigma_{j}(X) = 0$} for $j > n$. 
\end{definition}
In this paper, we relax the assumption that $\rank(F)$ is small and assume instead that the singular values of $F$ decay rapidly, so that $\erank(F)$ is small. Such scenarios occur in numerical computing~\cite{fortunato2017fast,  townsend2016computing, wilber2017computing}, where~\eqref{eq:Sylv} arises from the discretization of certain PDEs and $F$ is associated with a smooth 2D function (see Section~\ref{sec:poisson}). Current methods only bound singular values of $X$ with indices that are multiples of $\rho$ (see Section~\ref{sec:ADI}), and therefore provide little to no information in this setting.  To overcome this issue, we derive bounds on the singular values of $X$ that depend on the singular values of $F$ directly.  Our results are rooted in two fundamental observations:

\vspace{-.2cm}

\begin{itemize}[leftmargin=*,noitemsep]
\item {\bf Splitting property:} Equation~\eqref{eq:Sylv} can be split into $\rho$ matrix equations, each with a rank 1 right-hand side. Specifically, \smash{$X = \sum_{i = 1}^\rho X_i$}, and each \smash{$X_i$} satisfies
\begin{equation}
\label{eq: rank1}
AX_i - X_i B = \sigma_i(F)u_iv_i^*, 
\end{equation}
where  \smash{$\sum_{i = 1}^{\rho} \sigma_i(F) u_iv_i^*$} is the singular value decomposition (SVD) of $F$. 
\item {\bf Bounding property:} Bounds on the singular values of \smash{$X_i$} exist that depend on the size of \smash{$\sigma_i(F)$} (see Section~\ref{sec:ExpBoundsF}).
\end{itemize}

\vspace{-.25cm} 

Since the majority of bounds in the literature are useful when $F$ is of rank 1~\cite{beckermann2016singular, sabino2007solution,penzl2000eigenvalue}, one can apply them to each of the equations in~\eqref{eq: rank1}, and then use the bounding property to exploit the decay of the singular values of $F$ (see Theorem~\ref{thm: circexp}).  Our underlying proof technique applies a modification of the alternating direction implicit (ADI) method to~\eqref{eq:Sylv} using the above two observations. In a worst-case example, this method constructs a near-best low rank approximation (see~\ref{sec:nearbest}). 

This paper is organized as follows: In Section~\ref{sec:ADI},  we use the factored ADI (fADI) method~\cite{benner2009adi} to derive known bounds when $\rank(F) \leq 2$. In Section~\ref{sec:BoundsF}, we develop a new method for bounding singular values when $F$ has rapidly decaying singular values. We discuss three examples in Section~\ref{sec:Examples}. Section~\ref{sec:FI-ADI} describes a practical method for solving~\eqref{eq:Sylv} in low rank form, and we apply this method to develop fast low rank Poisson solvers in Section~\ref{sec:poisson}.

\section{Explicit bounds on the singular values of matrices with displacement structure}
\label{sec:ADI}
Let $X$ satisfy~\eqref{eq:Sylv} with $m \geq n$, and suppose that $A$ and $B$ have no eigenvalues in common so that $X$ is the unique solution to~\eqref{eq:Sylv}~\cite{Sylvestermatrix}.  Also, assume that $A$ and $B$ are normal matrices.\footnote{Results for nonnormal matrices $A$ and $B$ are discussed in Section~\ref{sec:genFIADI}.}
Here, we show how the ADI and fADI methods can be used to derive bounds on the singular values of $X$.

\subsection{The ADI method} 
\label{sec: ADIiter}
 The ADI algorithm is an iterative method that numerically solves~\eqref{eq:Sylv} by alternately updating the column and row spaces of an approximate solution~\cite{lu1991solution, peaceman1955numerical}. 
 One ADI iteration consists of the following two steps:

\vspace{-.1cm} 
 
\begin{enumerate}[leftmargin=*,noitemsep]
\item Solve for \smash{$X^{(j+1/2)}$}, where 
\begin{equation}
\label{eq: stdADI}
\left( A - \beta_{j+1} I  \right)X^{(j+1/2)} =   X^{(j)}\left( B-\beta_{j+1} I \right)  +   F. 
\end{equation}
\item Solve for \smash{$X^{(j+1)}$}, where 
\begin{equation}
\label{eq: stdADI2}
X^{(j+1)}\left( B - \alpha_{j+1} I \right) = \left( A - \alpha_{j+1} I \right) X^{(j+1/2)}  - F.
\end{equation}
\end{enumerate}

An initial guess, \smash{$X^{(0)} =0$}, is required to begin the iterations. The construction of \smash{$X^{(k)}$} requires selecting a set of $k$ 2-tuples, \smash{$\{ ( \alpha_j, \beta_j)\}_{j =1}^k$}, referred to as \textit{shift parameters}. These parameters strongly affect the approximation error \smash{$\|X - X^{(k)}\|_2$} (see Section~\ref{sec:ADIerror}). 

\subsection{The fADI method}\label{sec:fADI_iter}
The fADI method~\cite{benner2009adi} is equivalent to ADI, but computes \smash{$X^{(k)}$} in low rank form. The fADI iteration is derived by expressing \smash{$X^{(j)}$} in terms of \smash{$X^{(j-1)}$} using~\eqref{eq: stdADI} and~\eqref{eq: stdADI2}, and then  substituting the factorizations \smash{$X^{(j)} = W^{(j)}D^{(j)}Y^{(j)^*}$} and \smash{$F = MN^*$}, where \smash{$M \in \mathbb{C}^{ m\times \rho }$} and \smash{$N \in \mathbb{C}^{n \times \rho}$}, into the resulting equation. After $k$ iterations, the following block matrices are constructed: 
\begin{align}
& W^{(k)} \!=\!\left[ \!\!\begin{array}{ c | c | c | c } \hat{W}^{(1)}\! &\!  \hat{W}^{(2)}\! &\! \cdots \!&\! \hat{W}^{(k)} \end{array}\!\! \right], \!\!\quad\!\! \label{eq:Z}
\begin{cases} 
\hat{W}^{(1)} = (A-\beta_1 I)^{-1}M, \\
\hat{W}^{(j+1)} = (A- \alpha_j I)(A-\beta_{j+1}I)^{-1}\hat{W}^{(j)},
\end{cases} \\
& Y^{(k)} \!=\! \left[ \!\!\begin{array}{ c | c | c | c }\hat{Y}^{(1)} \!&\!  \hat{Y}^{(2)} \!&\! \cdots \!&\! \hat{Y}^{(k)} \end{array}\!\! \right],\!\! \quad \!\!\begin{cases} \label{eq:Y}
\hat{Y}^{(1)} = (B^*-\overline{\alpha}_1 I)^{-1}N, \\
\hat{Y}^{(j+1)} = (B^*-\overline{\beta}_{j}I)(B^*- \overline{\alpha}_{j+1}I)^{-1}\hat{Y}^{(j)}, 
\end{cases}\\
& D^{(k)} = \textnormal{diag}\left(  (\beta_1 - \alpha_1)I_\rho , \ldots ,(\beta_k - \alpha_k)I_\rho \right) ,
\label{eq:D}
\end{align}
where \smash{$I_{\rho}$} is the $\rho \times \rho$ identity matrix. 
This method  is only computationally competitive when shifted linear solves involving $A$ and $B$ are cheap to perform, so that the total cost of the $2k\rho$ solves does not exceed the \smash{$\mathcal{O}(m^3 + n^3)$} cost of solving~\eqref{eq:Sylv} directly via the Bartels--Stewart algorithm~\cite[Ch.~7.6]{golub2012matrix}. 
 
Using fADI, one clearly sees that after $k$ iterations, the rank of the approximant \smash{$X^{(k)}$} is at most $ k\rho$.  By definition of \smash{$\sigma_j(X)$}, we conclude that
\begin{equation} 
\label{eq: ADIsvbnd}
\sigma_{k\rho+1}(X) \leq \| X - X^{(k)}\|_2, \qquad  0 \leq k\rho <  n.
\end{equation} 
The triples $(A, B, F)$ in~\eqref{eq:Sylv} satisfied by Cauchy, real-valued Vandermonde and positive definite Hankel matrices, among others, admit explicit bounds on \smash{$\| X - X^{(k)}\|_2$}~\cite{beckermann2016singular}.  Since $\rho \leq 2$ for these matrices, meaningful bounds on their singular values are supplied by~\eqref{eq: ADIsvbnd}. 
However,  if  $F$ is full rank ($\rho = n$),  no useful bound on the singular values of $X$ can be obtained from~\eqref{eq: ADIsvbnd}.  

\subsection{Bounding the ADI approximation error}\label{sec:ADIerror}
To find explicit bounds on the quantity \smash{$\|X - X^{(k)}\|_2$} in~\eqref{eq: ADIsvbnd}, we  examine 
the ADI error equation. It is given in~\cite{benner2009adi} as
\[
X - X^{(k)}  = r_k(A)(X -X^{(0)}) \rb{k}, \qquad  r_k(z) = \prod_{j = 1}^{k} \dfrac{ (z - \alpha_j)}{(z - \beta_j)}, \qquad k \geq 1.
\]
Assuming that  \smash{$X^{(0)}$} is chosen as the zero matrix, it follows that
 \begin{equation} 
\label{eq: ADIsvbndrat}
\|X-X^{(k)}\|_2 \leq \|r_k(A)\|_2 \; \|\rb{k}\|_2 \|X\|_2. 
\end{equation} 
We seek ADI shift parameters \smash{$\{ ( \alpha_j, \beta_j)\}_{j =1}^k$} that minimize \smash{$\|r_k(A)\|_2  \|\rb{k}\|_2$}.  Since $A$ and $B$ are normal matrices, we have that 
\begin{equation} 
\label{eq:bndspectrum}
\|r_k(A)\|_2 \; \| \rb{k} \|_2 \leq \sup_{z \in \lambda(A) } |r_k(z)| \sup_{ z \in  \lambda(B)} \dfrac{1}{| r_k(z) | },
\end{equation}
where $\lambda(A)$ and $\lambda(B)$ denote the spectra of $A$ and $B$, respectively. We replace the discrete sets $\lambda(A)$ and $\lambda(B)$ with closed regions $E$ and $G$ in the complex plane, where $\lambda (A) \subset E$ and $\lambda (B) \subset G$.  The optimal shift parameters are then described by the  rational  function that attains the following infimum: 
\begin{equation} 
\label{eq: rationalapprox}
Z_k(E, G): = \inf_{ r \in \mathcal{R}_k } \dfrac{ \sup_{z \in E} |r(z) | }{ \inf_{ z \in G} | r(z) |}.
\end{equation}  
Here, \smash{$\mathcal{R}_k$} is the set of rational functions of the form \smash{$r_k(z) = p_k(z)/q_k(z)$}, with polynomials \smash{$p_k$} and \smash{$q_k$}  of  degree at most $k$.\footnote{For all choices $E$ and $G$ considered in this paper, there is an extremal rational function $\tilde{r}_k = p_k/q_k$ that attains $Z_k$  with $p_k$ and $q_k$ of degree exactly $k$~\cite{akhiezer1990elements, starke1992near, zolotarev1877application}.}
 By~\eqref{eq:bndspectrum}, we have that  \smash{$Z_k(E, G)$}  bounds the ADI error as follows: 
\begin{equation} 
\label{eq:Zolobnd}
\|X-X^{(k)}\|_2 \leq Z_k(E,G) \|X\|_2. 
\end{equation}
Solutions  to~\eqref{eq: rationalapprox} and estimates on \smash{$Z_k(E,G)$} are known for certain choices of $E$ and $G$~\cite{ganelius1979rational, gonchar1969zolotarev, starke1992near, wachspress1962optimum, zolotarev1877application}. In some cases, \smash{$Z_k(E,G)$} can be expressed explicitly~\cite{beckermann2016singular} and \smash{$\|X - X^{(k)}\|_2$} is bounded outright.

\begin{figure} 
 \centering
 \begin{minipage}{.75\textwidth} 
 \centering
  \begin{overpic}[width=\textwidth]{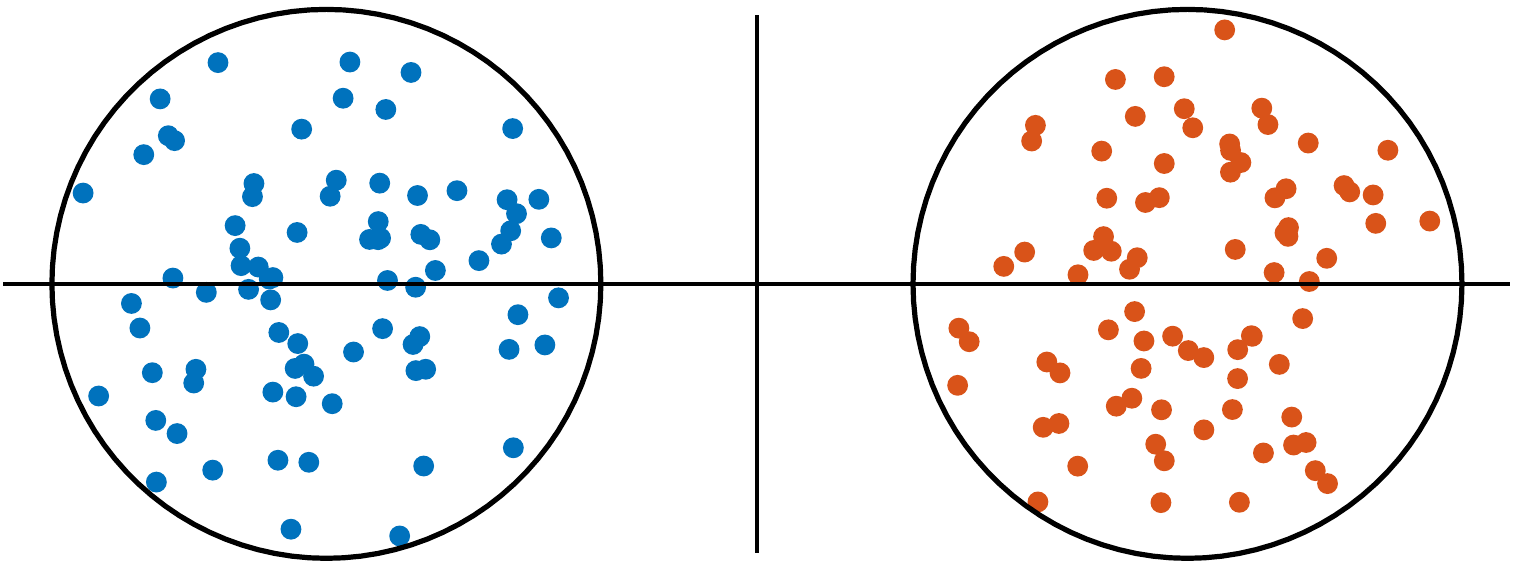}
  \put(48,37) {$Im$}
  \put(100, 18) {$Re$}
  \put (82, 25){\rotatebox{45}{ $\eta$ }}
  \put (25, 25){\rotatebox{45}{ $\eta$ }}
  \put (77, 16){$z_0$}
  \put(78, 18) {\tiny $|$}
  \put(18, 16){$-z_0$}
  \put(21.1, 18) {\tiny $|$}
  \put (86, 5) {\large $E$}
  \put (10, 5) {\large$-E$}
  \put(78.5,19){\color{black}\vector(1,1){12.5}}
  \put(21.5,19){\color{black}\vector(1,1){12.5}}
  \end{overpic}
  \end{minipage}
  \caption{The value of $Z_k(E, -E)$ is known~\cite{starke1992near} when $E = \{ z \in \mathbb{C} : |z - z_0| \leq \eta \},$
with $0 < \eta < z_0$. These results  are used in Theorem~\ref{thm: circexp} to bound the singular values of $X$ satisfying $AX - XB = F$, where $\lambda(A) \subset E$ and $\lambda(B) \subset -E$.}
  \label{fig:Cauchytwodisk}
  \end{figure}
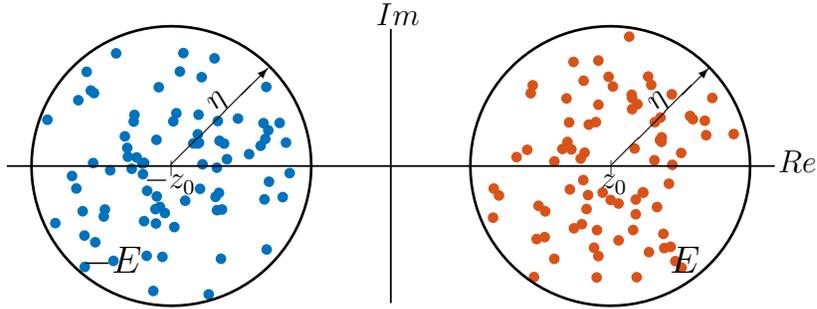

\subsection{Explicit bounds on singular values} 
\label{sec:example1}
We now illustrate how to explicitly bound the singular values of matrices using the fADI method.  Let \smash{$C \in \mathbb{C}^{m \times n}$} be a Cauchy matrix, $m \geq n$, with entries \smash{$C_{ij} = 1/(z_i - w_j) $,} where \smash{$ \{ z_i\}_{i = 1}^{m}$} and \smash{$ \{ w_j\}_{j = 1}^{n}$} are distinct collections of complex numbers, with \smash{$ \{ z_i\}_{i = 1}^{m}$} contained in  \smash{$E: = \{  z \in \mathbb{C} :  |z - z_0| \leq \eta \}, $} $ 0 < \eta < z_0$, \smash{$z_0, \eta \in \mathbb{R}$}, and \smash{$ \{ w_j\}_{j = 1}^{n} \subset -E$.} The matrix $C$ satisfies  
\begin{equation} 
\label{eq: Cauchysylv}
D_zC - CD_w= \vec{1}, 
\end{equation}
where \smash{$D_z = {\rm diag}( z_1, \ldots,  z_m)$, $D_w={\rm diag}( w_1,  \ldots,  w_n)$}, and $\vec{1}$ is the rank 1 $m \times n$ matrix of all ones.  The eigenvalues of \smash{$D_z$} and \smash{$D_w$} lie in the disks $E$ and $-E$, respectively (see Figure~\ref{fig:Cauchytwodisk}).  Since $C$ has a displacement rank of $ \rank(\vec{1}) = 1$, we have from~\eqref{eq: ADIsvbnd} and~\eqref{eq:Zolobnd} that
\begin{equation} 
\label{eq:CbndZ}
\sigma_{k + 1}(C) \leq Z_k(E, -E) \|C\|_2, \qquad 0 \leq  k <  n.
\end{equation} 
To bound \smash{$Z_k(E, -E)$}, we apply the following result from~\cite[Sec. 2]{starke1992near}: 
\begin{theorem} 
\label{thm:equidisk}
Let \smash{$E = \{ z \in \mathbb{C} : |z - z_0| \leq \eta \}$}, \smash{$0 < \eta < z_0$}, \smash{$z_0, \eta \in \mathbb{R}$}. 
Then, the infimum in~\eqref{eq: rationalapprox} is attained by the rational function
 \begin{equation} 
 \label{eq:circleshifts}
 \tilde{r}_k(z) = \left( \dfrac{ z - \phi }{z + \phi } \right)^k, \qquad \phi = \sqrt{ z_0^2 - \eta^2},
 \end{equation}
and \smash{$Z_k(E, -E)$} is given by 
 \begin{equation}
 \label{eq:mu1}
 Z_k(E, -E) = \mu_1 ^{-k}, \qquad \mu_1  = \dfrac{ z_0 +  \phi  } { z_0 - \phi  }.
 \end{equation}
\end{theorem}
\begin{proof}
For a proof, see Theorem 3.1 in~\cite{starke1992near} and the related discussion. 
\end{proof}

Using Theorem~\ref{thm:equidisk}, we have that for $C$ in~\eqref{eq: Cauchysylv}, 
\begin{equation} 
\label{eq:Cauchybnd}
 \sigma_{k+1}(C)  \leq \mu_1 ^{-k} \|C\|_2, \qquad 0 \leq k < n.
\end{equation} 
This shows that the singular values of $C$ decay at least geometrically. A consequence of~\eqref{eq:Cauchybnd} is that for $0 < \epsilon < 1$, $\erank(C) \leq  \lceil \log(1/\epsilon) / \log(\mu_1 )\rceil$. As discussed in~\cite{beckermann2016singular}, this method can be used to find explicit bounds on singular values for several classes of matrices with low $(A,B)$-displacement rank. 

\section{Explicit bounds on the singular values of matrices with high displacement rank}\label{sec:BoundsF}
The approach in Section~\ref{sec:example1} can be uninformative. To see this, consider \smash{$\tilde{C} \in \mathbb{C}^{m \times n}$}, $m \geq n$,  with entries  \smash{$\tilde{C}_{ij} = 1/ |z_i - w_j|^2$}. The matrix \smash{$\tilde{C}$}  satisfies  
\begin{equation}
\label{eq:bars}
\overline{D}_z \tilde{C} - \tilde{C}  \overline{D}_w = C, 
\end{equation}
where \smash{$\overline{M}$} denotes entrywise complex conjugation on $M$, and $C$, \smash{$D_z$} and \smash{$D_w$} are as in~\eqref{eq: Cauchysylv}. The singular values of $C$ have rapid decay. However, the displacement rank of \smash{$\tilde{C}$} is \smash{$\rank(C) = n$}, so the bound in~\eqref{eq: ADIsvbnd} only applies to \smash{$\sigma_1(\tilde{C})$}.  This reveals nothing about whether \smash{$\tilde{C}$} has low numerical rank.  Figure~\ref{fig:Cauchy2bounds} (left) shows that the singular values of \smash{$\tilde{C}$}  decay rapidly, and we require a new approach to bound them explicitly. 

\subsection{Bounds via a modification of Smith's method}\label{sec:Smiths}
The optimal ADI shift selection strategy for solving~\eqref{eq:bars} uses the same shift parameters,  \smash{$\alpha_j = \phi$} and \smash{$\beta_j =  -\phi$}, where $\phi$ is given in~\eqref{eq:circleshifts}, at every iteration.  When this happens, the fADI method is equivalent to Smith's method~\cite{smith1968matrix}. We first consider bounding singular values in this setting. 

Eq.~\eqref{eq:bars} is a special case of~\eqref{eq:Sylv}, with \smash{$A = \overline{D}_z$} and \smash{$B = \overline{D}_w $} satisfying the assumptions in Theorem~\ref{thm:equidisk}, and $F = C$.  Applying $k$ iterations of fADI to~\eqref{eq:bars} constructs an approximant \smash{$X^{(k)} = W^{(k)}D^{(k)}Y^{(k)^*}$}, where the factors are given by~\eqref{eq:Z},~\eqref{eq:Y}, and~\eqref{eq:D}. The dimensions of \smash{$W^{(k)}$} and \smash{$Y^{(k)}$} are $m \times k\rho$ and $n \times k \rho$, respectively.  When $\rho = n$, it is often the case that these matrices have linearly dependent columns, and this leads to an overestimation of  \smash{$\erank(X)$.} However, in applying $k$ iterations of fADI to~\eqref{eq:Sylv}, several potential low rank approximants to \smash{$X$} have been generated in addition to \smash{$X^{(k)}$}.
To see this, write  \smash{$X^{(k)}$} as a sum of $k\rho$ rank 1 terms,
\begin{equation} 
\label{eq:fADIrank1}
X^{(k)} = \sum_{ i = 1}^{\rho} \sum_{j = 1}^{k} \underbrace{d_{ij} \vec{w}_{ij}  \vec{y}_{ij} ^*}_{=T_{ij}},
\end{equation}
where \smash{$\vec{w}_{ij}$} and \smash{$\vec{y}_{ij}$} are the $i$th columns of  the blocks \smash{$\hat{W}^{(j)}$} and \smash{$\hat{Y}^{(j)}$}, respectively, in~\eqref{eq:Z} and~\eqref{eq:Y},  and \smash{$d_{ij}$} is the $(i,i )$ entry of \smash{$\hat{D}^{(j)}$} in~\eqref{eq:D}. The sum in~\eqref{eq:fADIrank1} exactly recovers the solution $X$  in the limit as $k \to \infty$. 
 
We now represent $X$ by arranging the rank 1 terms in~\eqref{eq:fADIrank1} in a $\rho \times \infty$ rectangle $\mathcal{R}$, so that each  \smash{$T_{ij}$} is represented by the box in the $i$th row and $j$th column of $\mathcal{R}$ (see Figure~\ref{fig:fADI}). An approximant  can be constructed by choosing any finite collection of boxes and summing together the terms that they represent.  For example, the fADI algorithm constructs \smash{$X^{(k)}$} by summing together the terms represented in the first $k$ columns of $\mathcal{R}$, as shown in Figure~\ref{fig:fADI} (left). A natural question to ask is whether this is the best choice. 
 
To answer this question, we examine the error associated with these approximants. If \smash{$\tilde{X}_t$} is constructed from a collection \smash{$\mathcal{K}_t$} of $t$ boxes in $\mathcal{R}$, then  \smash{$\|X - \tilde{X}_t\|_2$} is bounded above by \smash{$\sum_{\{(i,j) \in \mathcal{R}\setminus\mathcal{K}_t\}} \|T_{ij}\|_2$}.  To approximately minimize the error,  we choose \smash{$\mathcal{K}_t$} by selecting terms in decreasing order of their norms. Careful examination of the fADI method reveals that \smash{$\|T_{ij}\|_2$} is influenced by \smash{$Z_{j-1}(E, -E)$} and $\sigma_i(F)$: In~\eqref{eq:Z} and~\eqref{eq:Y}, $F$  is written as \smash{$MN^*$}. Assign \smash{$M = U\Sigma$} and $N = V$, where \smash{$U \Sigma V^*$} is the SVD of $F$.  It follows that 
\begin{equation} 
\label{eq:Tnorms}
\|T_{ij}\|_2 \leq \dfrac{\phi}{2(z_0-\eta)^2}  \sigma_i(F)Z_{j-1}(E, -E), 
\end{equation}  
where $\phi$ and \smash{$\tilde{r}_{j-1}(z)$} are given in~\eqref{eq:circleshifts}.

Consider \smash{$\tilde{C}$} in~\eqref{eq:bars}. In this case,  \smash{$Z_{j-1}(E, -E)= \mu_1 ^{-(j-1)}$} by Theorem~\ref{thm:equidisk}. The right-hand side of~\eqref{eq:bars} is the matrix $C$ in~\eqref{eq: Cauchysylv}, so it follows from~\eqref{eq:Cauchybnd} that \smash{$\|T_{ij}\|_2 \leq \phi \mu_1 ^{-(i+j-2)}\|C\|_2 /(2(z_0-\eta)^2)$}. This suggests that we construct \smash{$\tilde{X}_t$} by selecting rank 1 terms along the antidiagonals of \smash{$\mathcal{R}$} (see Figure~\ref{fig:fADI} (right)). 
This strategy leads to bounds on the singular values of \smash{$\tilde{C}$} with indices that do not depend on $\rho=n$, since \smash{$\rank(\tilde{X_t})$} is at most $k(k+1)/2$, as opposed to \smash{$k \rho$},  and \smash{$\sigma_{k(k+1)/2 +1}(\tilde{C}) \leq \|\tilde{C} - \tilde{X}_t\|_2$}. The same reasoning  applies for any matrix $X\in\mathbb{C}^{m\times n}$ satisfying~\eqref{eq:Sylv}, where $A$ and $B$ are as in Theorem~\ref{thm:equidisk} and \smash{$\sigma_i(F)\leq \mu_1^{-(i-1)}\|F\|_2$}. 

\begin{figure} 
\begin{minipage}{.40\textwidth}
\begin{tikzpicture}
\node[] at (-.6,3.5) {$\sigma_1(F)$};
\node[] at (-.6,2.5) {$\sigma_2(F)$};
\node[] at (-.6,1.5) {$\sigma_3(F)$};
\node[] at (-.6,0.5) {$\sigma_4(F)$};
\draw[black,->, ultra thick](0,0)--(5.5*.9,0);
\draw[black, ultra thick] (0,0)--(0,4);
\draw[black] (4*.9,0)--(4*.9,4);
\draw[black,->, ultra thick](0,4)--(5.5*.9,4);
\draw[black] (1*.9,0)--(1*.9,4);
\draw[black] (2*.9,0)--(2*.9,4);
\draw[black] (3*.9,0)--(3*.9,4);
\draw[black] (4*.9,0)--(4*.9,4);
\draw[black,->] (0,1)--(5.5*.9,1);
\draw[black,->] (0,2)--(5.5*.9,2);
\draw[black,->] (0,3)--(5.5*.9,3);
\draw[black,->] (0,4)--(5.5*.9,4);
\draw[black] (5*.9,0)--(5*.9,4); 
\node[] at (2.5,4.3) {no.~of fADI steps};
\node[] at (.5*.9,.5) (a) {4};
\node[] at (.5*.9,1.5) (a) {3};
\node[] at (.5*.9,2.5) (a) {2};
\node[] at (.5*.9,3.5) (a) {1};
\node[] at (1.5*.9,.5) (a) {8};
\node[] at (1.5*.9,1.5) (a) {7};
\node[] at (1.5*.9,2.5) (a) {6};
\node[] at (1.5*.9,3.5) (a) {5};
\node[] at (2.5*.9,.5) (a) {12};
\node[] at (2.5*.9,1.5) (a) {11};
\node[] at (2.5*.9,2.5) (a) {10};
\node[] at (2.5*.9,3.5) (a) {9};
\node[] at (3.5*.9,.5) (a) {16};
\node[] at (3.5*.9,1.5) (a) {15};
\node[] at (3.5*.9,2.5) (a) {14};
\node[] at (3.5*.9,3.5) (a) {13};
\node[] at (4.5*.9,.5) (a) {\textcolor{lightgray}{20}};
\node[] at (4.5*.9,1.5) (a) {\textcolor{lightgray}{19}};
\node[] at (4.5*.9,2.5) (a) {\textcolor{lightgray}{18}};
\node[] at (4.5*.9,3.5) (a) {\textcolor{lightgray}{17}};
\end{tikzpicture} 
\end{minipage} 
\hspace{1cm}
\begin{minipage}{.40\textwidth}
\begin{tikzpicture}
\node[] at (-.6,3.5) {$\sigma_1(F)$};
\node[] at (-.6,2.5) {$\sigma_2(F)$};
\node[] at (-.6,1.5) {$\sigma_3(F)$};
\node[] at (-.6,0.5) {$\sigma_4(F)$};
\draw[black,->, ultra thick](0,0)--(5.5*.9,0);
\draw[black, ultra thick] (0,0)--(0,4);
\draw[black] (4*.9,0)--(4*.9,4);
\draw[black,->, ultra thick](0,4)--(5.5*.9,4);
\draw[black] (1*.9,0)--(1*.9,4);
\draw[black] (2*.9,0)--(2*.9,4);
\draw[black] (3*.9,0)--(3*.9,4);
\draw[black] (4*.9,0)--(4*.9,4);
\draw[black] (5*.9,0)--(5*.9,4);
\draw[black,->] (0,2)--(5.5*.9,2);
\draw[black,->] (0,3)--(5.5*.9,3);
\draw[black,->] (0,4)--(5.5*.9,4);
\draw[black,->] (0,1)--(5.5*.9,1);
\node[] at (2.5,4.3) {no.~of fADI steps};
\node[] at (.5*.9,.5) (a) {10};
\node[] at (.5*.9,1.5) (a) {6};
\node[] at (.5*.9,2.5) (a) {3};
\node[] at (.5*.9,3.5) (a) {1};
\node[] at (1.5*.9,.5) (a) {\textcolor{lightgray}{14}};
\node[] at (1.5*.9,1.5) (a) {9};
\node[] at (1.5*.9,2.5) (a) {5};
\node[] at (1.5*.9,3.5) (a) {2};
\node[] at (2.5*.9,.5) (a) {\textcolor{lightgray}{18}};
\node[] at (2.5*.9,1.5) (a) {\textcolor{lightgray}{13}};
\node[] at (2.5*.9,2.5) (a) {8};
\node[] at (2.5*.9,3.5) (a) {4};
\node[] at (3.5*.9,.5) (a) {\textcolor{lightgray}{22}};
\node[] at (3.5*.9,1.5) (a) {\textcolor{lightgray}{17}};
\node[] at (3.5*.9,2.5) (a) {\textcolor{lightgray}{12}};
\node[] at (3.5*.9,3.5) (a) {7};
\node[] at (4.5*.9,.5) (a) {\textcolor{lightgray}{26}};
\node[] at (4.5*.9,1.5) (a) {\textcolor{lightgray}{21}};
\node[] at (4.5*.9,2.5) (a) {\textcolor{lightgray}{16}};
\node[] at (4.5*.9,3.5) (a) {\textcolor{lightgray}{11}};
\end{tikzpicture} 
\end{minipage} 
\caption{ 
The box in the $i$th row and $j$th column represents the rank 1 term \smash{$T_{ij}$} from~\eqref{eq:fADIrank1}. The terms reduce in norm as one applies successive ADI iterations (moving to the right), but they also reduce in norm as the index of the singular values of $F$ are increased (moving down). In this illustration, we suppose that \smash{$\|T_{ij}\|_2 = \mathcal{O}(\mu_1 ^{-(i+j-2)})$} and \smash{$\rank(F) = 4$.}  Left: With $k = 4$, the fADI algorithm constructs   \smash{$X^{(k)}$}, where  \smash{$ \rank(X^{(k)}) \leq k^2= 16$}, by summing terms represented by the first \smash{$k$} columns of the rectangle. The numbering of the boxes designates the order in which the rank~1 terms are constructed via fADI; decay in the singular values of $F$ is not exploited. Right: The boxes are numbered in decreasing order with respect to their norms.  Only the first \smash{$ t = k(k+1)/2$} terms (numbered in black) are required to construct an approximant \smash{ $\tilde{X}_t$ }so that  \smash{ $\|X - \tilde{X}_t\|_2 \approx \|X - X^{(k)}\|_2$}.  }
\label{fig:fADI}
\end{figure} 

\subsubsection{Explicit bounds on singular values}\label{sec:ExpBoundsF}
We now require explicit bounds on expressions of the form \smash{$\|X - \tilde{X}_t\|_2$}. We find them using the splitting and bounding properties from Section~\ref{sec:Introduction}.

\vspace{-.25cm}

\begin{itemize}[leftmargin=*,noitemsep]
\item {\bf Applying the splitting property.} The strategy depicted in Figure~\ref{fig:fADI} (right) is equivalent to splitting~\eqref{eq:Sylv} into $\rho$ equations and applying a different number of fADI iterations to each one. The $i$th row of  $\mathcal{R}$  corresponds to the $i$th equation in~\eqref{eq: rank1}. Applying \smash{$s_i$} iterations of fADI to~\eqref{eq: rank1} results in \smash{$X_i^{(s_i)}$}, where \smash{$\| \sum_{j = s_i+1}^{\infty} T_{ij} \|_2= \|X_i - X_i^{(s_i)}\|_2$}.  The sum of these errors bounds the total error \smash{$\|X - \tilde{X}_t\|_2$}, where  \smash{$\tilde{X}_t = \sum_{i = 1}^{\rho} X_i^{(s_i)}$} and \smash{$t = \sum_{i = 1}^\rho s_i$}.   
\item {\bf Applying the bounding property.} For each \smash{$X_i$}, we have a bound of the form \smash{$\|X_i - X_i^{(s_i)}\|_2 \leq Z_{s_i}(E, -E) \|X_i\|_2$}. To find a bound that explicitly involves the singular value \smash{$\sigma_i(F)$}, we use the following result:
\end{itemize} 

\begin{lemma} 
\label{thm:Horn}
Let $X\in\mathbb{C}^{m\times n}$ satisfy $AX-XB=F$ for normal matrices $A$ and $B$. Further, suppose that $\lambda(A) \subset E$ and $\lambda(B) \subset -E$, where $E$ is the disk \smash{$E:= \{ z \in \mathbb{C} : |z - z_0| \leq \eta \},$} with $z_0, \eta\in\mathbb{R}$ and $0 < \eta < z_0$.
Then, 
\[
\|X\|_2 \leq \dfrac{ \|F\|_2}{2(z_0 - \eta)}.
\]
\end{lemma}
\begin{proof}
The lemma follows as a special case of~\cite[Thm.~2.1]{horn1998two}. 
\end{proof}

Applying Lemma~\ref{thm:Horn} to~\eqref{eq: rank1}, we find that \smash{$\|X_i\|_2 \leq \sigma_i(F)/(2(z_0 - \eta))$.}
Using this result, we can now derive explicit bounds on the singular values of $X$. We begin with the case where \smash{$\sigma_{k}(F)$} decays at the same rate as \smash{$Z_k(E, -E)$.}

 \begin{theorem} 
 \label{thm: circexp}
 Let \smash{$X \in \mathbb{C}^{m \times n}$}, $m \geq n$,  satisfy $AX -XB = F$,  with $\lambda(A) \subset E$ and $ \lambda(B) \subset -E$ , where \smash{$E = \{ z \in \mathbb{C} : |z - z_0| \leq \eta \},$} with \smash{$z_0, \eta \in \mathbb{R}$} 
and \smash{$0 < \eta < z_0$.}
 Suppose that for \smash{$0 \leq j < n$}, 
 \smash{$ \sigma_{j+1}(F) \leq K  \mu_1 ^{-j} \|F\|_2,$}   where  $\mu_1$ is given in~\eqref{eq:mu1} and $K \geq 1$ is a constant. 
  For the  triangular numbers $1 \leq t =k(k+1)/2 < n$,  the singular values of $X$ are bounded in the following way: 
  \begin{equation} 
  \label{eq:THM4bnd}
  \sigma_{t+1} (X) \leq  K \dfrac{z_0 + \eta}{z_0-\eta}     \;  (\tfrac{3}{2} \sqrt{t} + 1) \mu_1 ^{-( \sqrt{8t+1}-1)/2} \; \|X\|_2.
  \end{equation}
  \end{theorem}
\begin{proof}
Let \smash{$\rank(F) = \rho$}. Consider the approximant 
\smash{$ \tilde{X}_t = \sum_{i = 1}^{k} \sum_{ j = 1}^{k +1 -i} T_{ij}, $}
where \smash{$T_{ij}$} are given in~\eqref{eq:fADIrank1}. We allow the choice $k > \rho$ with the convention that for $s > \rho $, \smash{$\|T_{sj}\|_2 = 0$}.\footnote{ For expository reasons, when $k > \rho$, we do not account for the non-contribution of the terms \smash{$\|T_{sj}\|_2 = 0$} in our bounds. This  simple but notationally tedious task would improve the bounds associated with $k > \rho$.} This corresponds to selecting terms along the antidiagonals of \smash{$\mathcal{R}$} in Figure~\ref{fig:fADI}. Since \smash{$\rank(\tilde{X}_t) \leq  t =   k(k+1)/2$}, we have that \smash{$\sigma_{t+1}(X) \leq \|X - \tilde{X}_t\|_2$}.  The proof proceeds by bounding the approximation error \smash{$\|X - \tilde{X}_t\|_2$}. The error equation  is given by
\[
X - \tilde{X}_t = \underbrace{\sum_{i = k +1}^{\rho} \sum_{j = 1}^{\infty} T_{ij}}_{=S_1}+ \underbrace{\sum_{i = 1}^{k } \sum_{j = k +1-i}^{\infty} T_{ij}}_{=S_2} .
\]
Using the fact that \smash{$\sum_{j = 1}^\infty T_{ij} = X_i$}, where $X_i$ is given in~\eqref{eq: rank1}, we find that \smash{$S_1$} satisfies \smash{$AS_1 - S_1B = \sum_{i = k +1}^{\rho} \sigma_{i}(F)u_iv_i^*$}. 
It follows from Lemma~\ref{thm:Horn} that 
\begin{equation} 
\label{eq:bound1}
  \| S_1 \|_2 \leq  \dfrac{\sigma_{k+1}(F)}{ 2(z_0- \eta)}\leq \dfrac{K \|F\|_2  \mu_1 ^{-k }}{ 2(z_0-\eta) }.
 \end{equation}

To bound \smash{$\|S_2\|_2$}, observe that \smash{$S_2=  \sum_{i = 1}^{k } (X_i - X_i^{(s_i)})$}, where \smash{ $X_i^{(s_i)}$} is constructed by applying \smash{$s_i=k\!+\!1\! -\! i$} steps of fADI to~\eqref{eq: rank1}.  For each $i$, we have
\[
\|X_i - X_i^{(s_i)}\|_2\leq Z_{s_i}(E, -E) \; \|X_i\|_2 \leq  \dfrac{ \sigma_i(F)}{  2(z_0-\eta) }  \mu_1 ^{-s_i},
\]
where Lemma~\ref{thm:Horn} has been used to bound \smash{$\|X_i\|_2$}.  This implies that
\begin{equation} 
\label{eq:bound2} 
\|S_2\|_2 \leq  \dfrac{K \|F\|_2}{  2(z_0-\eta)}  \sum_{i = 1}^k \mu_1 ^{-(i-1) - s_i } \leq \dfrac{ K \|F\|_2 }{2(z_0-\eta) } k \mu_1 ^{-k }, 
\end{equation}
and~\eqref{eq:bound1} and~\eqref{eq:bound2} together give the bound
\begin{equation} 
\label{eq: bound3}
\sigma_{t+1}(X) \leq \|X - \tilde{X}_t\|_2 \leq \dfrac{K \|F\|_2}{2(z_0- \eta)} (k+1) \mu_1 ^{-k }.  
\end{equation} 
To get a relative bound, we must divide the expressions in~\eqref{eq: bound3} by \smash{$\| X\|_2$}.  Trivially, the relation $AX - XB = F$ implies that \smash{$  1/ \|X\|_2 \leq \left( \|A\|_2 + \|B\|_2 \right) / \|F\|_2$}.  Due to the assumptions on $E$ we have \smash{$\|A\|_2+ \|B\|_2 \leq 2(z_0+\eta)$}. The theorem follows from  the fact that \smash{$k = (\sqrt{ 8t + 1} -1)/2$}, and for $t \geq 1$, \smash{$k \leq 3 \sqrt{t}/2$}.
\end{proof}

In Theorem~\ref{thm: circexp}, it is assumed for convenience that $t$ is a triangular number. However, for any  $1 \leq t <  n$,  a bound on \smash{$\sigma_{t+1}(X)$} is found by bounding the sum of the first $t$ terms selected along the antidiagonals of $\mathcal{R}$ (see Figure~\ref{fig:fADI} (right)). The constants in~\eqref{eq:THM4bnd} are due to estimates on \smash{$\|X\|_2$}, and are therefore not necessarily tight. However, as shown in~\ref{sec:nearbest}, there are $A$, $B$ and $F$ satisfying Theorem~\ref{thm: circexp} so that for $1 \leq t \leq \rho(\rho+1)/2$,   \smash{$\|X - \tilde{X}_t\|_2 \approx \sigma_{t+1}(X)$.}  This implies that \smash{$\tilde{X}_t$} is a near-best low rank approximation to $X$, and that the decay rate \smash{$\mu_1^{-( \sqrt{8t+1}-1)/2}$} in~\eqref{eq:THM4bnd} cannot be improved without additional assumptions on  $A$, $B$, and $F$ (see~\ref{sec:nearbest}).   

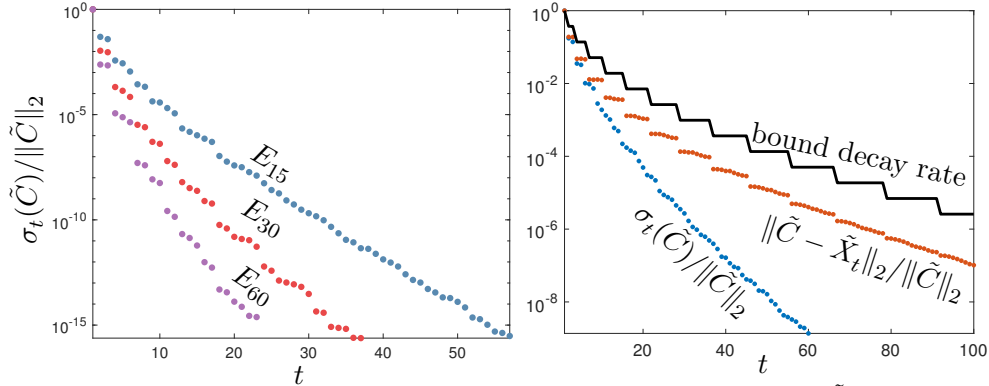
\begin{figure}
\centering
 \begin{minipage}{.45\textwidth}
   \begin{overpic}[width=\textwidth]{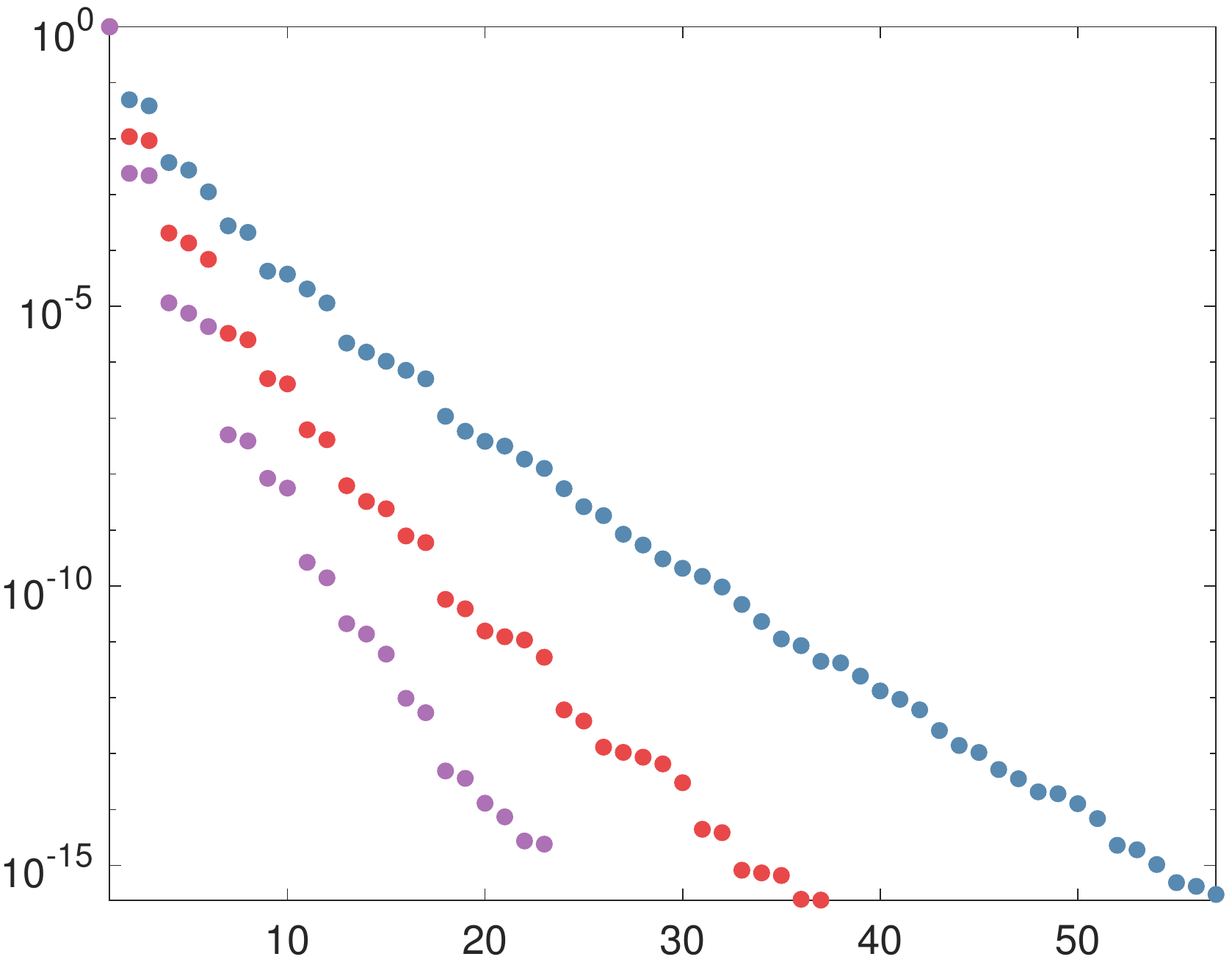}
  \put(-9,23) {\rotatebox{90}{{ \small $\sigma_t(\tilde{C})/\|\tilde{C}\|_2$}}}
  \put(50,-5) {{ \small $t$}}
  \put(38,33) {\rotatebox{-28}{{ \small $E_{30}$}}}
  \put(40,45) {\rotatebox{-30}{{ \small $E_{15}$}}}
  \put(36,19) {\rotatebox{-33}{{ \small $E_{60}$}}}
 \end{overpic}
 \end{minipage}
\begin{minipage}{.45\textwidth}
  \begin{overpic}[width=\textwidth]{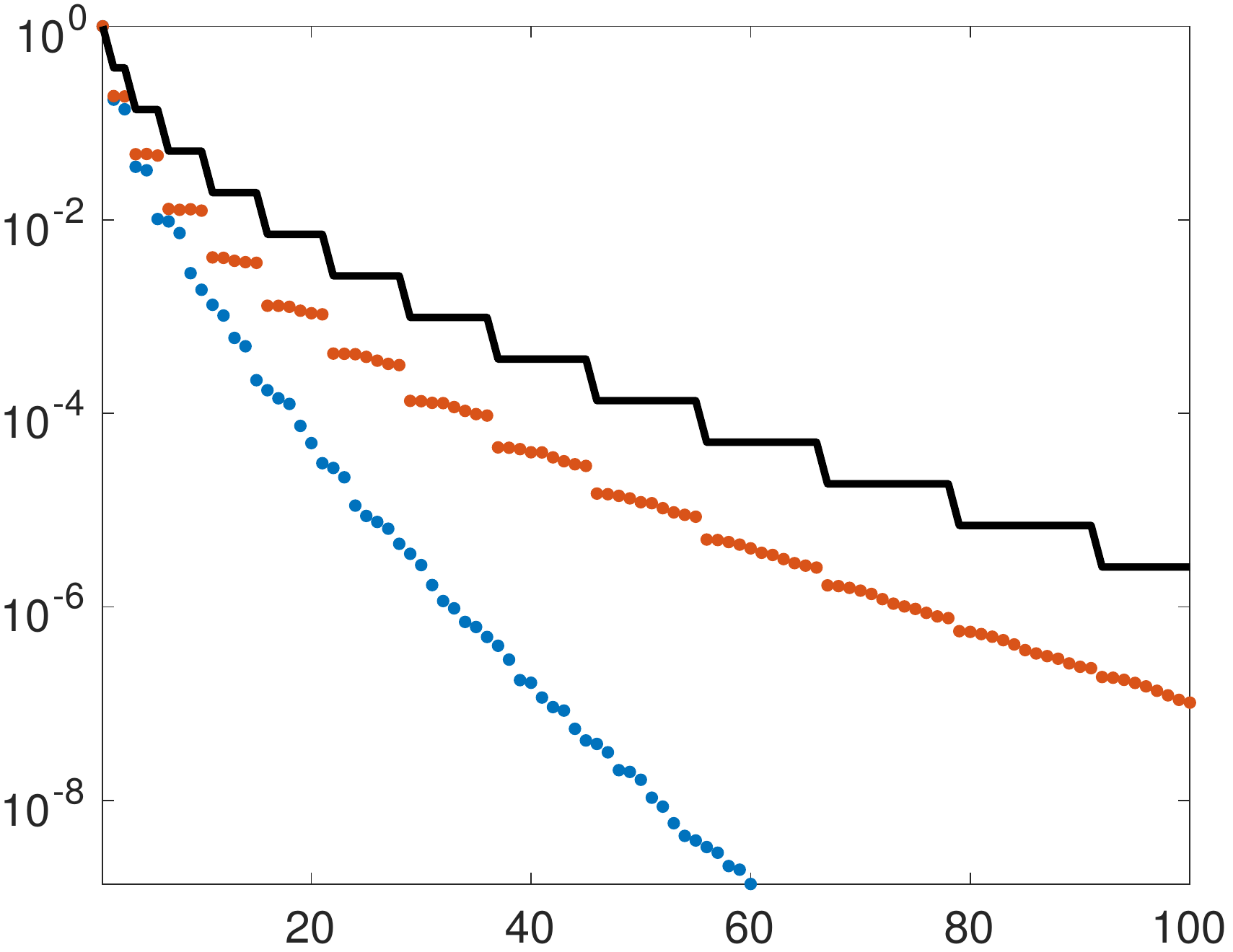}
  \put(50,-3) {\small{$t$}}
  \put(48,47){\rotatebox{-12}{\small bound decay rate}}
  \put(22,31){\rotatebox{-43}{\small $\sigma_t(\tilde{C})/\|\tilde{C}\|_2$}}
  \put(50,27){\rotatebox{-18}{\small $\|\tilde{C} - \tilde{X}_t\|_2/\|\tilde{C}\|_2$}}
 \end{overpic}
 \end{minipage}
 \caption{Left: The normalized singular values for matrices of the form \smash{$\tilde{C}_{ij} = 1/|z_i - w_j|^2$} are plotted for three different selections of sets \smash{$\{ z_i \}_{i = 1}^{100} \subset E_{\gamma}$} and  \smash{$\{ w_j \}_{j = 1}^{100} \subset -E_\gamma$} for $\gamma = 15$ (blue), $30$ (red),  and $60$ (purple), where \smash{$E_\gamma$} is a disk of radius $10$ with center $(\gamma, 0)$. Rapid decay of the singular values is observed. Right: The error from constructing a rank \smash{$t$} approximant as described in Section~\ref{sec:Smiths} (red) bounds the normalized singular values for \smash{$\tilde{C}$} of size \smash{$1000 \times 1000$} (blue), and is bounded above by the decay rate (excluding the polynomial factor) from the bound in Theorem~\ref{thm: circexp}  (black). } 
 \label{fig:Cauchy2bounds}
\end{figure}

Applying Theorem~\ref{thm: circexp} to \smash{$\tilde{C}$} shows that for triangular numbers  $t$, $1 \leq t < n$, 
\begin{equation} 
\label{eq:cauchy2bnds}
 \sigma_{t+1} (\tilde{C}) \leq  \dfrac{ z_0 + \eta}{z_0-\eta}     \;  (\tfrac{3}{2} \sqrt{t} + 1) \mu_1 ^{-( \sqrt{8t+1}-1)/2} \; \|\tilde{C}\|_2.
\end{equation}
Figure~\ref{fig:Cauchy2bounds}  displays the decay rate from~\eqref{eq:cauchy2bnds}, as well as the error \smash{$\|\tilde{C} - \tilde{X}_t\|_2 / \|\tilde{C}\|_2$}, where \smash{$\tilde{X}_t$} is constructed as in Theorem~\ref{thm: circexp}. 
These results also give a bound on \smash{$\erank(\tilde{C})$}. For $0 < \epsilon < 1$, we have that
\begin{equation} 
\label{eq:circerank2}
\erank(\tilde{C}) \leq \dfrac{  k^*(k^*+1)}{2}, \qquad  k^* = \left \lceil \log \left(  \dfrac{(z_0 + \eta) (\tfrac{3}{2}\sqrt{n}+1)}{(z_0 - \eta)  \epsilon} \right ) \bigg/ \log\mu_1  \right \rceil, 
\end{equation}
where we have used the fact that \smash{$\sqrt{t} \leq \sqrt{n}$.} 
For fixed \smash{$ 0 < \epsilon < 1$}, the bound in~\eqref{eq:circerank2} only grows polylogarithmically with $n$, so that for very large $n$, standard operations, such as matrix-vector multiplication, can be performed to an \smash{$\epsilon$}-accuracy in quasi-optimal computational complexity by using \smash{$\tilde{X}_t$}.

We need not require that the decay rate of $\sigma_i(F)$ matches the decay rate of \smash{$Z_k(E, -E)$.} As an example, suppose that $\sigma_i(F)$ decays with $i$ at a geometric rate twice that of \smash{$Z_k(E, -E)$}, i.e., \smash{$\sigma_{i+1}(F) \leq K \mu_1 ^{-2i} \|F\|_2 $.}  It is no longer optimal to construct an approximant \smash{$\tilde{X}_t$} by selecting terms along antidiagonals  of $\mathcal{R}$ (see Figure~\ref{fig:fADI} (right)). Instead, the number of fADI iterations applied to each $X_i$ in~\eqref{eq: rank1} (each row of $\mathcal{R}$) must be modified.  Specifically, for each \smash{$X_i$},  construct \smash{$X_i^{(s_i)} $} with \smash{$s_i = 2k - 2(i-1)$} to form  \smash{$\tilde{X}_t = \sum_{i = 1}^k X^{(s_i)}$.}   If, on the other hand,  \smash{$\sigma_{i+1}(F) \leq K  \mu_1 ^{-i/2} \|F\|_2$},  then \smash{$\tilde{X}_t$} is constructed by performing \smash{$s_i = k\!+\!1\!-\!i$} iterations of fADI on  \smash{$X_i$} and \smash{$X_{i+1}$ }simultaneously. 
Generalizing from these examples, we have the following corollary:
\begin{corollary} 
\label{cor: thm4corr}
Suppose that the assumptions of Theorem~\ref{thm: circexp} hold, except that  \smash{$\sigma_{i+1}(F) \leq K   \mu_F^{-i} \|F\|_2$}, with \smash{$\mu_F > 1$}.   Let \smash{$\mu = \min ( \mu_F, \mu_1  )$}, and define the integer $\ell$ as  \smash{$\ell = \left \lfloor  \log \left( \max ( \mu_F, \mu_1  ) \right) / \log \mu \right \rfloor$}.  Then, for the numbers  $1 \leq t =\ell k(k+1)/2 < n$, the singular values of $X$ are bounded as 
\[
\sigma_{t+1} (X) \leq  K \dfrac{ z_0 + \eta}{z_0-\eta}     \;  (\tfrac{3}{2}\sqrt{t} + 1) \mu^{-\ell( \sqrt{8t+1}-1)/2} \; \|X\|_2.
\]
\end{corollary}

Further generalizations of Theorem~\ref{thm: circexp} hold whenever explicit bounds are known for the singular values of $F$, even if the rate of decay is not geometric. For example, with the same assumptions as Theorem~\ref{thm: circexp} except that the singular values of $F$ decay algebraically, the singular values of $X$ can be shown to decay at the same algebraic rate. 

A generalization of Theorem~\ref{thm: circexp} and Corollary~\ref{cor: thm4corr} can be stated when $E$ and $G$ are any two closed disks in the complex plane that are disjoint from each other. This follows from Theorem 3.1 in~\cite{starke1992near}, where \smash{$Z_k(E, G)$} and \smash{$\tilde{r}_k$} are given for disks $E$ and $G$ that are each symmetric about the real axis, as well as the observation that \smash{$Z_k(E, G)$} is invariant under rotation.

\subsection{Bounds via a modification of fADI} 
We now consider $AX - XB = F$, with $\lambda(A) \subset [-b,-a]$, $\lambda(B) \subset [a,b]$ for  $a,b\in\mathbb{R}$ and $0 < a < b$.  This scenario arises, for example, in the discretization of Poisson's equation (see Section~\ref{sec:poisson}). The extremal rational function $\tilde{r}_k$ that attains the infimum $Z_k([-b,-a], [a,b])$  in~\eqref{eq: rationalapprox} is known~\cite{akhiezer1990elements,wachspress1962optimum,zolotarev1877application}, and an upper bound on $Z_k([-b,-a], [a,b])$ is given in~\cite{beckermann2016singular} as 
\begin{equation} 
\label{eq: Zbndinterval}
Z_k([-b,-a],[a,b]) \leq 4\mu_2^{-k}, \qquad \mu_2 = \exp\left( \dfrac{\pi^2}{ \log(4 b/a)} \right) .
\end{equation} 
The zeros and poles of  \smash{$\tilde{r}_k(z)$}  can be computed using elliptic integrals~\cite{lu1991solution, zolotarev1877application}, and they form a set of $k$ ADI shift parameters \smash{$\{(\alpha_{\ell}, \beta_{\ell})\}_{\ell = 1}^k$.} In contrast to Section~\ref{sec:Smiths}, one cannot expect that the extremal function \smash{$\tilde{r}_{j}(z)$} that attains the infimum $Z_j([-b,-a],[a,b])$ has any zeros or poles in common with \smash{$\tilde{r}_k(z)$} when $j\neq k$. To use explicit bounds associated with \smash{$Z_j([-b,-a],[a,b])$} for $1 \leq j \leq k$,  we must allow for the use of several sets of shift parameters when constructing our ADI-based approximant \smash{$\tilde{X}_t$}. 
This is a natural generalization of the approach used in Theorem~\ref{thm: circexp}, and it leads to the following theorem: 
\begin{theorem} 
 \label{thm: intervalexp}
 Let \smash{$X \in \mathbb{C}^{m \times n}$}, $m \geq n$, satisfy $AX-XB = F$, and suppose that the assumptions in Theorem~\ref{thm: circexp} hold, except that $\lambda(A) \subset [-b,-a]$ and $\lambda(B) \subset [a,b]$, with $a,b\in \mathbb{R}$ and $0 < a < b$. 
 Then, for the triangular numbers $ 1 \leq t =  k(k+1)/2 < n$, we have
\[
\sigma_{t + 1}(X) \leq \dfrac{ K b}{a} (6 \sqrt{t} +1) \mu_2^{-  (\sqrt{ 8t + 1} -1)/2} \|X\|_2.
\]
 \end{theorem}
\begin{proof} 
Let \smash{$X = \sum_{i = 1}^\rho X_i$,} where \smash{$\rank(F)  = \rho$} and each \smash{$X_i$} satisfies~\eqref{eq: rank1}. For each $i \leq k$, construct the approximant
\smash{$\tilde{X}_t = \sum_{i = 1}^k  X_i^{(s_i)},$} \smash{$s_i = k \! + \! 1\!  - \! i,$} where \smash{$X_i^{(s_i)}$} is constructed by applying \smash{$s_i = k \! +\! 1\! - \! i$} iterations of fADI to~\eqref{eq: rank1} using optimal ADI shift parameters (these parameters are different for each $i$).  It follows that \smash{$\sigma_{t + 1}(X) \leq \|X - \tilde{X}_t\|_2$}, and the proof consists of bounding the error \smash{$\|X - \tilde{X}_t\|_2 $}.  This can be done just as in Theorem~\ref{thm: circexp} if one uses the fact that $[a, b]$ can be contained in a disk with radius $\eta = (b-a)/2$ and center $z_0 = (b+a)/2$, so that Lemma~\ref{thm:Horn} is applicable.
 \end{proof}
As before, we have the following corollary when the singular values of $F$ and  $Z_k([-b,-a],[a,b])$ decay at potentially different geometric rates: 
 \begin{corollary} 
\label{cor: thm6corr}
Suppose that the assumptions of Theorem~\ref{thm: intervalexp} hold, except that  \smash{$\sigma_{i+1}(F) \leq K \mu_F^{-i} \|F\|_2$}, with \smash{$\mu_F > 1$}.   Let \smash{$\mu = \min ( \mu_F, \mu_2 )$}, and define  $\ell$ as  \smash{$\ell = \left \lfloor  \log \left( \max ( \mu_F, \mu_2) \right) / \log \mu \right \rfloor$}.  Then, for  \smash{$1 \leq t =\ell k(k+1)/2 < n$}, we have 
  \begin{equation} 
 \label{eq:inte_explicitbndsL} 
 \sigma_{t + 1}(X) \leq \dfrac{K b}{a} (6 \sqrt{t} +1) \mu^{- \ell ( \sqrt{ 8t + 1} -1)/2} \|X\|_2.
\end{equation} 
\end{corollary}
\begin{proof} 
For a sketch of the proof, see the discussion preceding Corollary~\ref{cor: thm4corr}.
 \end{proof}
 
Related bounds can be stated when  $\lambda(A) \subset [a, b]$ and $\lambda(B) \subset [c, d]$, with $a < b < c < d$. In this case, with $F$ as in Corollary~\ref{cor: thm6corr},  we find that
\[
\sigma_{t + 1}(X) \leq  K \dfrac{{ \rm max}(|a|, |b|) +{ \rm max}(|c|, |d|)  }{|c-b|} (6 \sqrt{t} + 1) \mu^{- \ell (\sqrt{ 8t + 1} -1)/2} \|X\|_2,
\]
where $t$ is as in Corollary~\ref{cor: thm6corr}, $\mu = \min ( \mu_F,  \exp( \pi^2/ \log 16 \gamma ) )$, and $\gamma$ is the cross-ratio $|c-a| \; |d-b| /(  |c-b| \; |d-a|) $. This result is found by using a M{\"o}bius transformation that preserves \smash{$Z_k([a,b], [c,d])$} to map $[a,  b] \cup [c, d]$ to symmetric intervals $[-\alpha, -1] \cup [1, \alpha]$ (see~\cite{beckermann2016singular}), and  then applying Corollary~\ref{cor: thm6corr}.

\section{Examples}\label{sec:Examples}
The ideas and results in Section~\ref{sec:BoundsF} are connected to and inform a variety of other results. We give three examples that show how the splitting and bounding properties can be used. 

\subsection{The Hadamard product with a Cauchy matrix}\label{sec:CHP}
Let $A$, $B$ and $F$ satisfy the assumptions of Theorem~\ref{thm: circexp}.\footnote{Analogous results hold under the assumptions of Theorem~\ref{thm: intervalexp}, as well as the various generalizations of these theorems.}  Since $A$ and $B$ are normal matrices, they have eigendecompositions   \smash{$A = Y \Lambda_A Y^*$} and \smash{$B = W \Lambda_B W^*$.} Therefore,  $X$ in~\eqref{eq:Sylv} can be written  in closed form as
\begin{equation} 
\label{eq:hadamardsol}
 X  = Y \left(  C \circ ( Y^* F W) \right) W^*,
\end{equation}
where $C$ is a Cauchy matrix with entries \smash{$ C_{jk} = 1/( (\Lambda_A)_{jj}-(\Lambda_B)_{kk}) $}, and \smash{`$\circ$'} is the Hadamard  matrix product. Bounds on the singular values of $X$ can be determined using~\eqref{eq:hadamardsol}, rather than the method in Section~\ref{sec:BoundsF}. First, we split $AX -XB = F$ into the $\rho = \rank(F)$ equations  in~\eqref{eq: rank1}.
By~\eqref{eq:hadamardsol}, each \smash{$X_i$} in~\eqref{eq: rank1} can be expressed as
\begin{equation} 
\label{eq: splitHC}
X_i =  \sigma_i(F)  Y(C \circ (  Y^* u_iv_i^* W ))W^*. 
\end{equation}
 For each $i \leq k$, we use fADI on~\eqref{eq: Cauchysylv} to construct a rank \smash{$ \leq s_i = k \! + \! 1\! - \! i$}  approximant \smash{$C^{(s_i)}$}  to $C$. Substituting \smash{$C^{(s_i)}$} for $C$ in~\eqref{eq: splitHC} results in an approximant \smash{$X_{i}^{(s_i)}$}, and the sum of the matrices \smash{$X_i^{(s_i)}$} is an approximant to $X$.  This approach results in bounds of the form
\[ \sigma_{t +1}(X)  \leq 2 K\|C\|_2 (z_0 + \eta) ( \tfrac{3}{2} \sqrt{t} + 1) \mu_1 ^{-(\sqrt{8t+1}-1)/2} \|X\|_2 , \]
where \smash{$ 1 \leq t = k(k+1)/2  < n.$} This relates the singular values of $X$ to properties of the Cauchy matrix $C$. Generically, we have \smash{$\| C\|_2 \leq \sqrt{mn}/(2(z_0 - \eta))$} due to~\eqref{eq: Cauchysylv} and Lemma~\ref{thm:Horn}, but unfortunately, this does not result in bounds with an improved polynomial term when compared to~\eqref{eq:THM4bnd}. However, a more useful bound on $\|C\|_2$ may be available in specific cases.

This approach leads to an efficient algorithm for approximating $X$ in low rank form when fast matrix-vector products for $Y$ and $W$ are available (see Section~\ref{sec:Poisson} and also~\cite[Ch.~4.8]{golub2012matrix}).\footnote{To compute this approximant in low rank form, one uses the fact that for vectors $u_i = ( u_{1i}, \ldots, u_{mi})$ and $v_i$, 
\smash{ $ C \circ u_iv_i^*  = {\rm diag}(u_{1i}, \ldots, u_{mi}) C {\rm diag}(v_{1i}, \ldots, v_{ni})$}.}

\subsection{Families of structured matrices}
Let $C$ be a Cauchy matrix as in~\eqref{eq: Cauchysylv}  and define the family \smash{$\mathcal{F}_{m,n} =  \{ C^{\circ p} \}_{p = 1}^{\infty}$}, where \smash{$( C^{\circ p} )_{ij}  =  1/(z_i - w_j)^p$} and $m \geq n$.
For  $p \geq 2$, $C^{\circ p}$  satisfies the Sylvester equation \smash{$D_z C^{\circ p} -C^{\circ p}D_w = C^{\circ (p-1)}$}, and a recursive argument can be used to bound the singular values of each matrix in \smash{$\mathcal{F}_{m,n}$}. As an example, consider the matrix \smash{$C^{\circ 3}$}. For \smash{$C^{\circ 2}$}, Theorem~\ref{thm: circexp} can be applied directly, revealing that the singular values are bounded exactly as in~\eqref{eq:cauchy2bnds}. To bound the singular values of \smash{$C^{\circ 3}$}, define each \smash{$X_i$} so that it satisfies
\begin{equation}
\label{eq:c3breakup}
 D_z X_i-X_i D_w = \sum_{j = i(i+1)/2}^{ (i+1)(i+2)/2  -1} \sigma_j(C^{\circ 2}) u_jv_j^*,
 \end{equation}
 where \smash{$u_j$} and \smash{$v_j^*$} are the $j$th singular vectors of \smash{$C^{\circ 2}$}. The approximant \smash{$\tilde{X}_t$} to \smash{$C^{\circ 3}$} is constructed by applying \smash{$k \! + \!1\!-\! i$} fADI iterations to~\eqref{eq:c3breakup} for each $i \leq k$,  and then summing the resulting matrices. This is a variation on Theorem~\ref{thm: circexp} and results in a bound of the form 
 \begin{equation} 
 \label{eq:C3bnd}
 \sigma_{t+1}(C^{\circ 3}) \leq K_1  \mu_1 ^{-k } \|C^{\circ 3}\|_2 , \qquad 1 \leq t = \frac{1}{24} k (k+1)(k+2)(k+3)   < n,
 \end{equation} 
 where \smash{$K_1 = \mathcal{O}(\sqrt{n})$}.
It follows that as \smash{$n \to \infty$},  \smash{$\erank(C^{\circ 3}) = \mathcal{O} ( ( \log (\sqrt{n} / \epsilon ) ) ^4  )$.}  
As $p$ is increased, the bounds on the singular values of \smash{$C^{\circ p}$} become increasingly weak, but the bound on \smash{$\erank(C^{\circ p})$} always grows polylogarithmically with $n$.  This implies that for large enough $n$, the matrices in \smash{$\mathcal{F}_{m,n}$} are well-approximated by  low rank matrices. 
The set of $d$-dimensional, real-valued Vandermonde matrices satisfies a more complicated recursive relation that leads to similar bounds, and related results hold for various matrix families defined using the structured matrices in~\cite{beckermann2016singular}.

\subsection{A comparison to exponential sums}\label{sec:expsums}
Functions such as \smash{$f(x) = 1/x$} and \smash{$f(x) = \sqrt{x}$}, where $x \in [a, b] $,  and $0 \!< \!a < \!b \! < \!\infty$,  are well approximated by exponential sums of the form 
\[
S_k(x) = \sum_{j = 1}^k \alpha_j e^{ -t_j x},  \qquad \alpha_j, t_j \in \mathbb{R}.
\]   
In~\cite{braess2009efficient},  explicit bounds on the error  \smash{$\left \| 1/x - U_k(x) \right\|_{L_\infty ([a, b])}$} are given, where \smash{$U_k$} is the best approximation to \smash{$1/x$} by an exponential sum of at most $k$ terms. 
These estimates can be used to bound the singular values of \smash{$X \in \mathbb{R}^{n \times n}$}, where $X$ satisfies \smash{$AX + XA^T = BB^T$}, $\lambda(A) \subset [a,  b]$, and $A$ is a normal matrix~\cite{kressnerpersonal}. 

 Let \smash{$\hat{A} =  I_n \otimes A + A \otimes I_n$}, where \smash{`$\otimes$'} denotes the Kronecker matrix product. Then, \smash{${\rm vec}(X) = \hat{A}^{-1} {\rm vec}(BB^T)$,} where \smash{${\rm vec}(X) \in \mathbb{C}^{n^2 \times 1}$} is the column-major vectorization of $X$. We can use the fact that  \smash{$U_k(\hat{A}) \approx \hat{A}^{-1}$} to approximate $X$.  Results in~\cite{beckermann2016singular} and~\cite{braess2009efficient} give that 
\begin{equation}  
\label{eq:expsumbnd}
\left \| \hat{A}^{-1}- U_k(\hat{A}) \right\|_{2} \leq    \dfrac{K_\gamma k }{a}  \mu_2^{-k} ,  
\end{equation}
 where \smash{$K_\gamma$} is a bounded constant dependent on \smash{$\gamma = a/b$}, and \smash{$\mu_2$} is as in~\eqref{eq: Zbndinterval}.\footnote{In~\cite{braess2009efficient}, the left-hand side of~\eqref{eq:expsumbnd} is bounded by an expression involving the Gr\"otszch ring function~\cite[(19.2.8)]{olver2010nist}, which is associated with elliptic integrals. The more interpretable bound we give here follows from Theorem 3.1 in~\cite{beckermann2016singular}. } 

Let \smash{$\tilde{X}$} be the $n \times n$ matrix defined by ${\rm vec}(\tilde{X}) = U_k(\hat{A}){\rm vec}(BB^T)$. The property \smash{$\exp(  I_n \otimes A +  A \otimes I_n ) = \exp(A) \otimes \exp( A)$} can be used to show that \smash{$\rank(\tilde{X}) \leq  k \rho$}, where \smash{$\rank(BB^T) = \rho$}.  Since \smash{$\|X\|_F = \|{\rm vec}(X)\|_2$}, where \smash{$\| \cdot \|_F$} is the Frobenius norm, it follows that
\[ \sigma_{ k \rho+1}(X) \leq \left \| \hat{A}^{-1}{\rm vec}(BB^T) - U_k(\hat{A}) {\rm vec}(BB^T) \right \|_2 \leq \dfrac{K_\gamma k }{a} \mu_2^{-k} \; \| BB^T\|_F.\]
 To state a relative bound, we use the estimate \smash{$1/\|X\|_2 \leq 2b / \|BB^T\|_2$}, so that 
 \begin{equation} 
 \label{eq:expbondssvs}
 \sigma_{ k \rho+1}(X) \leq \tilde{K}_{\gamma} k  \mu_2^{-k} \; \| X\|_2, \quad  \tilde{K}_{\gamma} =  \dfrac{2 K_\gamma  b }{a} \dfrac{\|BB^T\|_F}{\|BB^T\|_2}.
 \end{equation} 
The bound attained via fADI is given  by~\cite[Cor. 4.1]{beckermann2016singular}, and states that
\begin{equation} 
\label{eq:ADIpick}
\sigma_{ k \rho+1}(X) \leq 4 \mu_2^{-k} \|X\|_2. 
\end{equation} 
The bounds in~\eqref{eq:expbondssvs} and~\eqref{eq:ADIpick} both achieve the same geometric decay rate, with the ADI-based bound resulting in a cleaner constant that does not include a factor of $k$. When $\rho$ is large, both bounds become problematic. 

However, the splitting and bounding properties can also be applied to improve bounds based on exponential sums~\cite{kressnerpersonal}.   For  $1\leq i \leq \rho$, let \smash{$X_i$} satisfy \smash{$AX_i +X_iA^T = \sigma_i(B)u_iu_i^T$},  and approximate each \smash{$X_i$} by the best sum of exponentials with at most \smash{$s_i$} terms. If the singular values of \smash{$BB^T$} decay at the geometric rate \smash{$\mu_2$}, then one should choose \smash{$s_i = k \!+ \!1 \!-i$}.  
It is clear from setting $\rho = 1$ in~\eqref{eq:expbondssvs} and~\eqref{eq:ADIpick} that bounds on the singular values of $X$ derived from exponential sums achieve the same decay rate as those in Theorem~\ref{thm: intervalexp}, but result in a larger algebraic factor involving $k$. 

\section{The FI-ADI method}\label{sec:FI-ADI}
The low rank approximations employed to bound singular values in Section~\ref{sec:BoundsF} can be automatically computed, resulting in an efficient method for approximately solving $AX - XB = F$ in low rank form  whenever $A$ and $B$ satisfy the assumptions in Theorem~\ref{thm: circexp} or Theorem~\ref{thm: intervalexp} (or their corollaries and generalizations), and  linear solves involving $A$ and $B$ can be performed cheaply (see Section~\ref{sec:poisson} for an application). We refer to this method as factored-independent ADI (FI-ADI).  An outline of the  FI-ADI method is given in the pseudocode below,\footnote{An implementation of FI-ADI in MATLAB is available at \url{https://github.com/ajt60gaibb/freeLYAP}.} where we assume the above conditions on $A$ and $B$ are met (see Sec~\ref{sec:genFIADI} for a generalization).  Key details for efficient implementation are described below.  

\begin{framed}
\begin{small}

\vspace{-.4cm}

 \begin{center} \textbf{\underline{The FI-ADI method}} \end{center}%

\vspace{-.25cm} 
 
\vspace{-.25cm} 
 
 \begin{description}[leftmargin=*] \setlength\itemsep{-.1cm}
  \item[] {\bf Input:}  $\circ$ $A \in \mathbb{C}^{m \times m}$, $B \in \mathbb{C}^{n \times n }$ , and $F$ satisfying $AX - XB =  F$, with \\ $F = \sum_{i = 1}^\rho \sigma_{i}(F) u_i v_i^*$
\item[]   \hspace{1.2cm} $\circ$ A tolerance $0 < \epsilon < 1$ 
\item[]   \hspace{1.2cm} $\circ$ Disjoint sets $E,G\subset\mathbb{C}$ such that $\lambda(A) \subset E$ and $\lambda(B) \subset G$
\item[] \hspace{1.2cm} $\circ$ A batch number $d$ and batching parameters $\{\ell_i\}_{i = 1}^{d+1}$, $\ell_{d+1} = \rho+1$.
\end{description} 

\vspace{-.25cm}

 \begin{description}[leftmargin=*] \setlength\itemsep{-.1cm}

\vspace{-.25cm} 

\item[] {\bf  Output:}  Factors $W$, $D$ and $Y$ satisfying \smash{$\|X - WDY^*\|_2 \approx \epsilon \|X\|_2$ }
\end{description}

\vspace{-.5cm} 

\begin{description}[leftmargin=*] \setlength\itemsep{-.1cm}
\item[1.] Split $\mathcal{S}:= AX - XB = \sum_{i = 1}^{\rho} \sigma_i(F) u_iv_i^*$ into $d$ equations: \\[-8pt]
\begin{equation} 
\label{eq:splitup}
\mathcal{S}_i : = AX_i - X_iB = \sum_{j = \ell_i}^{ \ell_{i+1}-1} \sigma_j(F) u_jv_j^*, \qquad  1 \leq i \leq d.
\end{equation} 

\vspace{-.3cm}

\item[2.] Find $\tau \approx \|X\|_2$ 
\item[3.] Set $W =[\phantom{z}]$, $D = [\phantom{z}]$, $Y = [\phantom{z}]$
 \item \hspace*{-.4cm} { \bf for}  $i = 1, \ldots, d$
\begin{description}[leftmargin=*] \setlength\itemsep{-.1cm} 
\item[(i)] Determine $s_i$ so that for $Z_{s_i}(E, G)$ in~\eqref{eq: rationalapprox},  
\vspace{-.3cm}
\begin{equation}
\label{eq:algbnd}
  Z_{s_i} (E, G) \leq \left( \epsilon \; \tau \; {\rm dist}(E, G)  \right) / \left(   d \; \sigma_{\ell_i}(F) \right), \quad {\rm dist}(E, G) = \min_{z \in E, w \in G}|z - w|
\end{equation}

\vspace{-.3cm}

\item[(ii)] Compute the set \smash{$\{ \alpha_{i,j}, \beta_{i,j}\}_{j = 1}^{s_i}$ }of optimal ADI shift parameters associated with $Z_{s_i}(E, G)$.
\item[(iii)] Apply $s_i$ steps of fADI to $\mathcal{S}_i$ to find $Z_i$, $D_i$ and $Y_i$
\item[(iv)] $W = \left[ \begin{array}{ c | c  } W & W_i  \end{array} \right]$, $Y= \left[ \begin{array}{ c | c  } Y & Y_i  \end{array} \right]$,  $D = { \rm diag}( D, D_i)$
\item[(v)] Compress $W$, $D$ and $Y$
\end{description}
\end{description}
\end{small}

\vspace{-.4cm}
\end{framed}

{\bf Error estimates.} As described in the pseudocode, the FI-ADI method constructs  \smash{$\tilde{X} = W D Y^*$.} If $\tau \leq \|X\|_2$, then \smash{ $\|X - \tilde{X}\|_2 \leq \epsilon \|X\|_2$}: by~\eqref{eq:Zolobnd}, Lemma~\ref{thm:Horn}, and the bound on \smash{$Z_{s_i}(E, G)$} in~\eqref{eq:algbnd},  \smash{$\|X_i - X_i^{(s_i)}\|_2 \leq ( \epsilon / d )\|X\|_2$}.
 A simple choice for $\tau$ is found using  \smash{ $\|F\|_2 \leq (\|A\|_2 + \|B\|_2) \|X\|_2$,}  but this  is often overly pessimistic. Settling for $\| X - \tilde{X}\|_2 \approx \epsilon \|X\|_2$, it is often more efficient to perform  a few steps of FI-ADI and then estimate $\tau$ using this approximant. We also find it effective in practice to choose the number of fADI steps for each $i$ as \smash{$s_i^* = \max( K_{max}, s_i)$}, where \smash{$s_i$} is computed as in Step (i) in the pseudocode and \smash{$K_{max}$} satisfies \smash{$Z_{K_{max}}(E, G) \leq \epsilon$. }

 {\bf Factorizations of $F$.} In Section~\ref{sec:BoundsF}, we used the SVD factorization $F = USV^*$ to derive bounds. This is also depicted in the pseudocode. However, the FI-ADI method works with any ``approximate SVD" of the form \smash{$F = \tilde{U} \tilde{\Sigma} \tilde{V}^*$}, where \smash{$\tilde{\Sigma}$} is diagonal with \smash{$\tilde{\Sigma}_{1,1} \geq  \cdots \geq \tilde{\Sigma}_{n,n}$}, and \smash{$\|\tilde{U}_{(:,i)}\tilde{V}_{(:,i)}^*\|_2 = 1$}. 

{\bf Computation of ADI shift parameters.} If $E$ and $G$ are disks in the complex plane, the required single shift parameter $(\alpha, \beta)$ is given by Theorem~\ref{thm: circexp}, a rotation mapping,  and the formula in~\cite{starke1992near}.  When $E$ and $G$ are closed real intervals, we refer the reader to the formulas in~\cite{lu1991solution}, as well as the MATLAB code in~\cite[Appendix~A]{fortunato2017fast}. For most  other choices of $E$ and $G$, heuristic shift selection strategies must be employed (see Section~\ref{sec:genFIADI}). 

{\bf Compression.}  The approximant \smash{$\tilde{X}$} is potentally a near-best low rank approximant to $X$ (see~\ref{sec:nearbest}), but in practice, \smash{$\tilde{X}$} can often be compressed. For large problems where memory is restrictive, an interim compression strategy (Step (v) in the pseudocode) is essential, and various schemes can be used (e.g.,~\cite{gugercin2003modified}). We apply the method from~\cite[Ch.~1.1.4]{bebendorf2008hierarchical}, where the skinny QR factorizations \smash{$ZD=Q_zR_z$} and \smash{$Y=Q_yR_y$} are used to find the truncated SVD of the small matrix \smash{$R_zR_y^*$.} 
The computational cost (in flops) of the compression step grows with the number of columns of $W$ as  \smash{$\mathcal{O} ( (m + n)t^2 + t^3)$}, where  $W$ has $t$ columns. It is beneficial to apply compression  after each iteration $i$ to keep the solution factors small. 

{\bf Batching linear solves.} Computational savings can be gained by grouping right-hand sides together when performing linear solves. For example, when the same ADI shift parameter  is used in every  fADI iteration for all \smash{$\mathcal{S}_i$} in~\eqref{eq:splitup}, the uncompressed factors \smash{$WDY^*$} are efficiently constructed  by applying $s_i$ fADI iterations to the equation \smash{$\sum_{j = 1}^{i} \mathcal{S}_j$} at each iteration $i$, with \smash{$s_{i-1} \geq s_{i}$}. 
Even when the shift parameters differ, efficiency is potentially  improved by grouping right-hand sides together in Step~1.  However, the cost of the compression step is also influenced by the batch sizes, and there is no simple choice of $d$ and \smash{$\{ \ell_i \}_{i = 1}^{d+1}$} that generically optimizes performance. 

{\bf FI-ADI versus fADI.} The FI-ADI method can be seen as a generalization of fADI, where more freedom has been permitted in the order that the rank 1 terms used to approximate $X$ are constructed. 
Using an  FI-ADI-based  method over fADI in  theoretical settings results in improved bounds on singular values of $X$ (see Section~\ref{sec:BoundsF}). However, the practical performance of either method depends greatly on implementation details, as well as the properties of $A$, $B$ and $F$. Vectorization and batched solves are  efficient, and fADI takes full advantage of this, whereas FI-ADI may not. The main practical benefit of FI-ADI is that re-ordering how rank 1 terms are constructed leads to an effective interim compression strategy.\footnote{A related idea is found in~\cite{gugercin2003modified}, where a reordering of terms (motivated by memory savings, not by the influence of the singular values of $F$) and an interim compression strategy results in an improved implementation of the low rank cyclic Smith method.}

\subsection{Generalized FI-ADI}
\label{sec:genFIADI}
We briefly review how an FI-ADI-based method can be used when the theorems and corollaries in Section~\ref{sec:BoundsF} are not applicable.  

{\bf Nonnormality.} Let $A$ and $B$ be diagonalizable but non-normal matrices, with eigendecompositions \smash{$A = V_A\Lambda_AV_A^{-1}$} and \smash{$B = V_B \Lambda_B V_B^{-1}$}.  The ADI error is bounded as
\smash{$ \|X-X^{(k)}\|_2 \leq \kappa_2(V_A) \kappa_2(V_B) \|r_k(\Lambda_A)\|_2 \; \|r_k(\Lambda_B)\|_2 \|X\|_2$,}
where \smash{$\kappa_2(M) = \|M\|_2  \|M^{-1}\|_2$}.  If bounds on \smash{$\kappa_2(V_A)$} and \smash{$\kappa_2(V_B)$} are known or can be numerically estimated, then the influence of these terms on the number of ADI steps can be estimated~\cite[Sec. 5]{townsend2016computing}.  Alternatively,  any spectral set~\cite{badea2013spectral} can be used to bound \smash{$\|r_k(A)\|_2\;  \|r_k(B)^{-1}\|_2$}~\cite[Cor. 2.2]{beckermann2016singular}.

{\bf Non-optimal shift selection.} If the sets $E$ and $G$ do not allow for optimal shift parameter selection, then one of many heuristic shift strategies may be applied~\cite[Ch.~4.4]{sabino2007solution}.  The use of suboptimal shifts affects convergence~\cite{simoncini2016computational}, and additional computational costs are incurred since either the ADI error equation or the residual equation must be monitored to determine convergence. We remark that only a few alternative schemes for solving $AX -XB = F$ when $\rank(F)$ is large have been proposed in the literature~\cite{kressner2016low, shank2014low}. When using FI-ADI, the residual error is  given by \smash{$\Delta_k(\mathcal{S}) : = \| r_k(A) F r_k(B)^{-1}\|_F$}~\cite{beckermann2011error}, where \smash{$\| \; \cdot \; \|_F$} is the Frobenius norm.  Using the submultiplicative property for \smash{$ \Delta_k(\mathcal{S}_i)$},  the influence of the singular values of $F$ can be exploited.

\section{A collection of low rank Poisson solvers} 
\label{sec:poisson}  In~\cite{fortunato2017fast}, spectral discretizations are developed so that the ADI method can be used to solve Poisson's equation on a variety of domains in optimal computational complexity (up to polylogarithmic factors). Combining these ideas with FI-ADI leads to highly efficient Poisson solvers that construct low rank approximations to solutions.  

\subsection{An FI-ADI--based Poisson solver on a square}\label{sec:Poisson}
Let $u$ be the solution to Poisson's equation on the square, i.e.,
\begin{equation} 
\label{eq:upoiss}
  \dfrac{\partial^2 u}{\partial x^2}  + \dfrac{\partial^2 u}{\partial y^2}  = f, \qquad x, y \in [-1, 1]^2, \qquad u(\pm 1,\cdot)  = u(\cdot, \pm 1)  = 0,
 \end{equation}
where $f$ is a smooth function on \smash{$[-1, 1]^2$}.  A standard numerical approach for finding $u$ involves discretizing~\eqref{eq:upoiss} using second-order finite differences.\footnote{This leads to a Lyapunov equation \smash{$D_2X + XD_2 = F$} that can be solved in only \smash{$\mathcal{O}(n^2 \log n)$} operations using the closed form solution in~\eqref{eq:hadamardsol} and the FFT~\cite[Ch.~4.8]{golub2012matrix}. Applying the FI-ADI method or using~\eqref{eq:hadamardsol} results in a fast low rank solver.} 
To achieve spectral accuracy, we instead apply the method in~\cite{fortunato2017fast}, where~\eqref{eq:upoiss} is discretized in a way that leads to the matrix equation \smash{$A \hat{X} - \hat{X} B= \tilde{D} \hat{F}\tilde{D}^T$}. Here, \smash{$\hat{X}$} and \smash{$\hat{F}$} contain scaled expansion coefficients for expressing $u$ and $f$, respectively, in a particular ultraspherical polynomial basis, and \smash{$\tilde{D}$} is diagonal. We refer the reader to~\cite[Sec.~3]{fortunato2017fast} for further details. This discretization is specifically designed for ADI-based approaches:  \smash{$A$} and \smash{$B$} satisfy the assumptions in Theorem~\ref{thm: intervalexp} and they are banded, so that linear solves involving them are cheap. For these reasons, FI-ADI is a highly efficient method for approximating \smash{$\hat{X}$} in low rank form. Recurrence relations among ultraspherical polynomials ensure that a rank $k$ approximation to $\hat{X}$ can be transformed to a convenient Chebyshev basis in \smash{$\mathcal{O}(k n \log n)$} operations~\cite{fortunato2017fast,olver2013fast}. The inverse transform is also fast, so that if $f$ is a smooth function, a low rank factorization of the matrix \smash{$\hat{F}$} can be found efficiently using methods in~\cite{townsend2014computing}.

The left panel of Figure~\ref{fig:poi} illustrates the computational savings gained from using the FI-ADI method to exploit the numerical rank of \smash{$\hat{X}$}.\footnote{ A faster implementation of both the FI-ADI and ADI-based solvers is achieved by performing the required linear solves with a subroutine written in C (see \url{https://github.com/danfortunato/fast-poisson-solvers}), which is not used here. The degrees of freedom in this experiment are increased artificially to demonstrate asymptotic complexity.} In this example, a matrix of bivariate Chebyshev coefficients for $f$ is given in low rank form for several choices of $f$. We use this to find $\hat{F}$ in low rank form. A low rank approximation to \smash{$\hat{X}$} is then computed and transformed to the Chebyshev basis. We compare this approach to the optimal complexity solver in~\cite{fortunato2017fast} that forms \smash{$\hat{F}$} explicitly, and then finds \smash{$\hat{X}$} in explicit form. 

The right panel displays a solution \smash{$\tilde{u}$} to~\eqref{eq:upoiss} computed in Chebfun~\cite{Chebfun} using this approach, where $f$ is smooth and its $512\times 512$ Chebyshev coefficient matrix  is approximated by a rank 206 matrix.  The exact solution is given by 
\[
u = (1-x^2)(1-y^2)  \sin( 3 \pi(1 + \cos( \pi x^2- \pi y^2 )) (x-2y)(2x+y) \cos( \pi x^2 + \pi y^2)).
\]
With the tolerance parameter set at \smash{$\epsilon = 10^{-10}$}, our approach results in an error of  \smash{$\|u -\tilde{u}\|_2/\|u\|_2 \approx  7.01 \times 10^{-11}.$} 

\begin{figure} 
 \centering
 \begin{minipage}{.46\textwidth} 
 \centering
  \begin{overpic}[width=\textwidth]{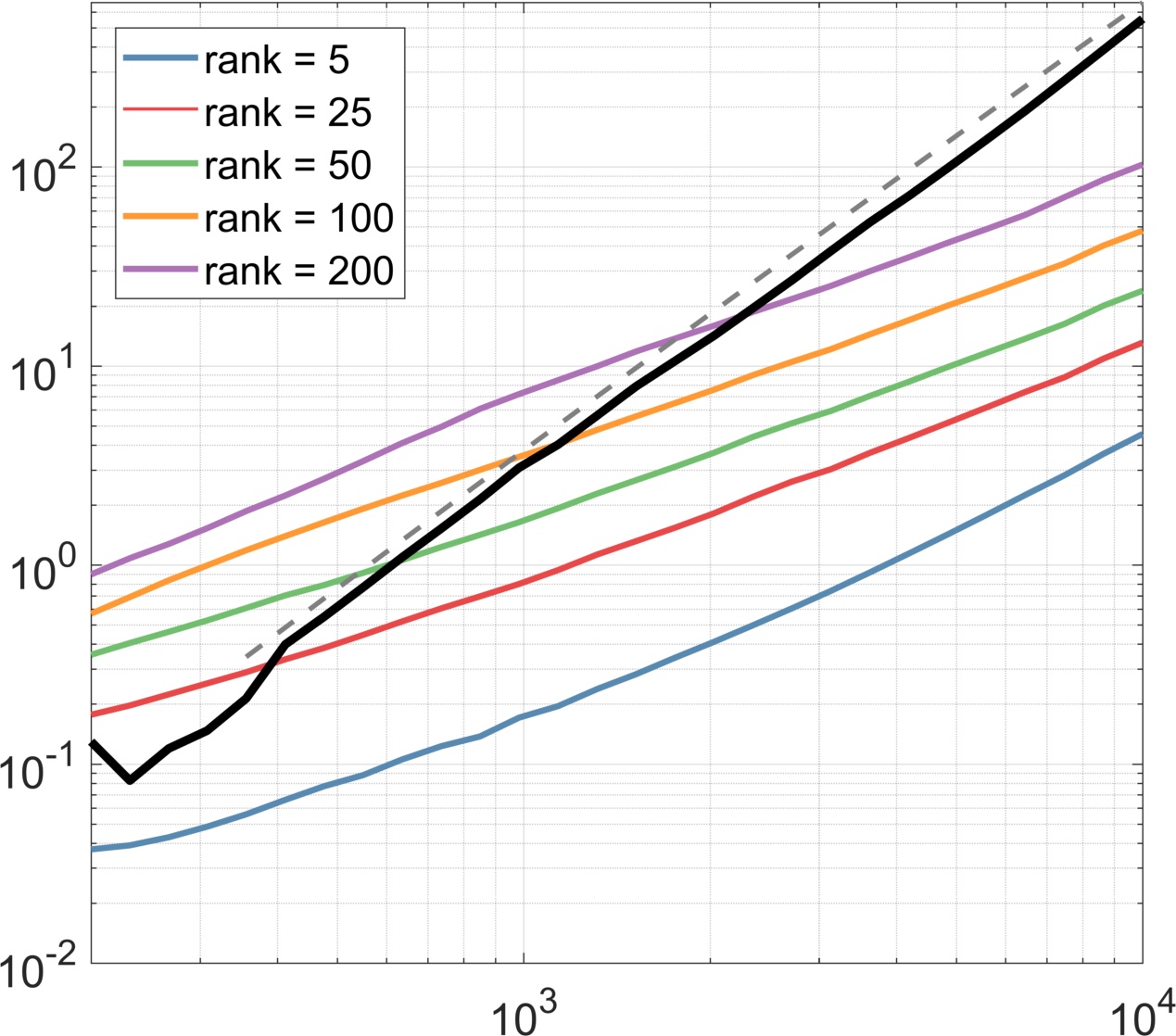}
  \put ( 46,-3){$n$}
  \put (54, 62){\rotatebox{36}{\small  $\mathcal{O}(n^2 (\log n )^2) $ }}
  \put(-6, 22){\rotatebox{90}{\small  execution time (s) }}
  \end{overpic}
  \end{minipage}
   \begin{minipage}{.44\textwidth} 
 \centering
  \begin{overpic}[width=.95\textwidth]{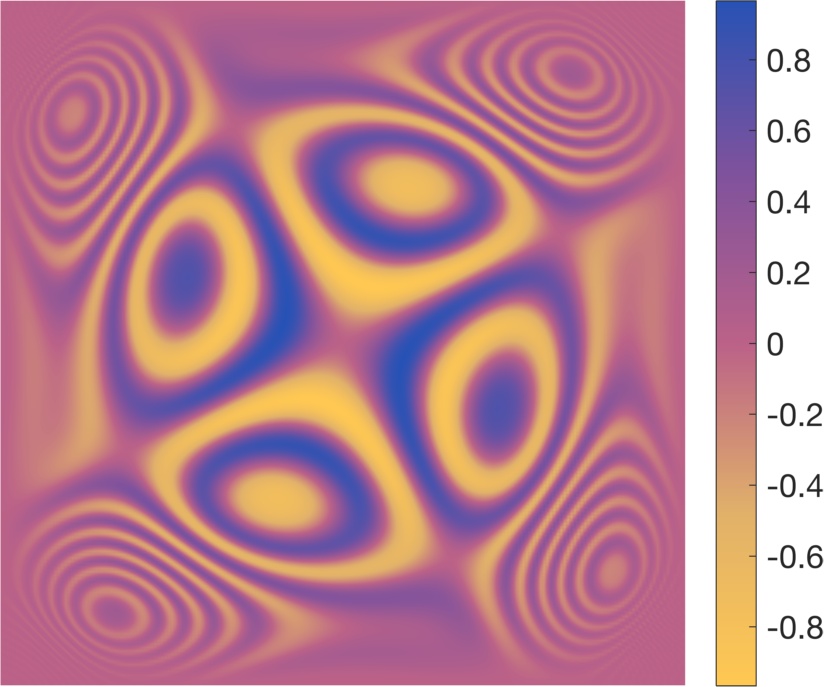}
  \end{overpic}
  \end{minipage}
  \caption{Left: The wall clock time in seconds is plotted against the problem size $n$ for computing the Chebyshev coefficients of an approximate solution to~\eqref{eq:upoiss}, with a relative tolerance of \smash{$1 \times 10^{-10}$}.  A low rank, FI-ADI--based solver with right-hand sides of varying rank (colored lines), all with rapidly decaying singular values, is compared against an optimal complexity solver that is unaffected by the rank of $F$ and returns $X$ in explicit form (black). The FI-ADI--based solver returns a low rank approximation  \smash{$WDY \approx X$}. Timings include compression of the factors $W$, $D$, and $Y$.  In both cases,  $F$ is provided in low rank form (with approximate singular values).   Right: The solution to~\eqref{eq:upoiss}, where $f$ is a smooth function with a \smash{$512 \times 512$} Chebyshev coefficient matrix $F$ of rank 206, computed with Chebfun using the spectral method in~\cite{fortunato2017fast} and FI-ADI. }
  \label{fig:poi}
  \end{figure}
  
\subsection{An FI-ADI--based Poisson for other domains}
This approach is not limited to a square domain: Combining FI-ADI with the discretizations described in~\cite{fortunato2017fast} and~\cite{wilber2017computing} leads to efficient low rank Poisson solvers for 2D functions on rectangles, disks, and on the surface of a sphere, and for 3D functions on solid spheres, cylinders, and cubes. We expect that similar solvers can be developed for many other elliptic PDEs. 

\section*{Acknowledgements} 
We thank Daniel Kressner for his correspondence with regards to Section~\ref{sec:expsums}. We are grateful to Dan Fortunato for  detailed comments on an early manuscript and code related to Section~\ref{sec:poisson}. We thank Anil Damle, Kathryn Drake, Marc Gilles, Daniel Kressner, and Nick Trefethen for their insightful responses to a draft, and we also thank an anonymous referee for many useful suggestions. 

\bibliographystyle{model1-num-names}
\bibliography{refs}

\begin{thebibliography}{38}
\expandafter\ifx\csname natexlab\endcsname\relax\def\natexlab#1{#1}\fi
\providecommand{\bibinfo}[2]{#2}
\ifx\xfnm\relax \def\xfnm[#1]{\unskip,\space#1}\fi
\bibitem[{Greengard and Rokhlin(1987)}]{greengard1987fast}
\bibinfo{author}{L.~Greengard}, \bibinfo{author}{V.~Rokhlin},
\newblock \bibinfo{title}{A fast algorithm for particle simulations},
\newblock \bibinfo{journal}{J. Comput. Phys.} \bibinfo{volume}{73}
  (\bibinfo{year}{1987}) \bibinfo{pages}{325--348}.
\bibitem[{Antoulas(2005)}]{antoulas2005approximation}
\bibinfo{author}{A.~C. Antoulas}, \bibinfo{title}{Approximation of large-scale
  dynamical systems}, \bibinfo{publisher}{SIAM}, \bibinfo{year}{2005}.
\bibitem[{Benner et~al.(2005)Benner, Mehrmann, and
  Sorensen}]{benner2005dimension}
\bibinfo{author}{P.~Benner}, \bibinfo{author}{V.~Mehrmann},
  \bibinfo{author}{D.~C. Sorensen}, \bibinfo{title}{Dimension reduction of
  large-scale systems}, volume~\bibinfo{volume}{45},
  \bibinfo{publisher}{Springer}, \bibinfo{year}{2005}.
\bibitem[{Cand{\`e}s and Recht(2009)}]{candes2009exact}
\bibinfo{author}{E.~J. Cand{\`e}s}, \bibinfo{author}{B.~Recht},
\newblock \bibinfo{title}{Exact matrix completion via convex optimization},
\newblock \bibinfo{journal}{Found. Comp. Math.} \bibinfo{volume}{9}
  (\bibinfo{year}{2009}) \bibinfo{pages}{717--772}.
\bibitem[{Beckermann and Townsend(2017)}]{beckermann2016singular}
\bibinfo{author}{B.~Beckermann}, \bibinfo{author}{A.~Townsend},
\newblock \bibinfo{title}{On the singular values of matrices with displacement
  structure},
\newblock \bibinfo{journal}{SIAM J. Matrix Anal. Appl.} \bibinfo{volume}{38}
  (\bibinfo{year}{2017}) \bibinfo{pages}{1227--1248}.
\bibitem[{Penzl(2000)}]{penzl2000eigenvalue}
\bibinfo{author}{T.~Penzl},
\newblock \bibinfo{title}{Eigenvalue decay bounds for solutions of {L}yapunov
  equations: the symmetric case},
\newblock \bibinfo{journal}{Systems \& Control Letters} \bibinfo{volume}{40}
  (\bibinfo{year}{2000}) \bibinfo{pages}{139--144}.
\bibitem[{Sabino(2007)}]{sabino2007solution}
\bibinfo{author}{J.~Sabino}, \bibinfo{title}{Solution of large-scale {L}yapunov
  equations via the block modified {S}mith method}, Ph.D. thesis, Rice
  University, \bibinfo{year}{2007}.
\bibitem[{Antoulas et~al.(2002)Antoulas, Sorensen, and
  Zhou}]{antoulas2002decay}
\bibinfo{author}{A.~C. Antoulas}, \bibinfo{author}{D.~C. Sorensen},
  \bibinfo{author}{Y.~Zhou},
\newblock \bibinfo{title}{On the decay rate of {H}ankel singular values and
  related issues},
\newblock \bibinfo{journal}{Systems \& Control Letters} \bibinfo{volume}{46}
  (\bibinfo{year}{2002}) \bibinfo{pages}{323--342}.
\bibitem[{Grasedyck(2004)}]{grasedyck2003existence}
\bibinfo{author}{L.~Grasedyck},
\newblock \bibinfo{title}{Existence and computation of low {K}ronecker-rank
  approximations for large linear systems of tensor product structure},
\newblock \bibinfo{journal}{Computing} \bibinfo{volume}{72}
  (\bibinfo{year}{2004}) \bibinfo{pages}{247--265}.
\bibitem[{Kressner and Uschmajew(2016)}]{kressner2016low}
\bibinfo{author}{D.~Kressner}, \bibinfo{author}{A.~Uschmajew},
\newblock \bibinfo{title}{On low-rank approximability of solutions to
  high-dimensional operator equations and eigenvalue problems},
\newblock \bibinfo{journal}{Linear Algebra Appl.} \bibinfo{volume}{493}
  (\bibinfo{year}{2016}) \bibinfo{pages}{556--572}.
\bibitem[{Simoncini(2016)}]{simoncini2016computational}
\bibinfo{author}{V.~Simoncini},
\newblock \bibinfo{title}{Computational methods for linear matrix equations},
\newblock \bibinfo{journal}{SIAM Review} \bibinfo{volume}{58}
  (\bibinfo{year}{2016}) \bibinfo{pages}{377--441}.
\bibitem[{Fortunato and Townsend(2017)}]{fortunato2017fast}
\bibinfo{author}{D.~Fortunato}, \bibinfo{author}{A.~Townsend},
\newblock \bibinfo{title}{Fast {P}oisson solvers for spectral methods},
\newblock \bibinfo{journal}{arXiv preprint arXiv:1710.11259}
  (\bibinfo{year}{2017}).
\bibitem[{Townsend et~al.(2016)Townsend, Wilber, and
  Wright}]{townsend2016computing}
\bibinfo{author}{A.~Townsend}, \bibinfo{author}{H.~Wilber},
  \bibinfo{author}{G.~B. Wright},
\newblock \bibinfo{title}{Computing with functions in spherical and polar
  geometries {I.} the sphere},
\newblock \bibinfo{journal}{SIAM J. Sci. Comput.} \bibinfo{volume}{38}
  (\bibinfo{year}{2016}) \bibinfo{pages}{C403--C425}.
\bibitem[{Wilber et~al.(2017)Wilber, Townsend, and
  Wright}]{wilber2017computing}
\bibinfo{author}{H.~Wilber}, \bibinfo{author}{A.~Townsend},
  \bibinfo{author}{G.~B. Wright},
\newblock \bibinfo{title}{Computing with functions in spherical and polar
  geometries {II.} the disk},
\newblock \bibinfo{journal}{SIAM J. Sci. Comput.} \bibinfo{volume}{39}
  (\bibinfo{year}{2017}) \bibinfo{pages}{C238--C262}.
\bibitem[{Benner et~al.(2009)Benner, Li, and Truhar}]{benner2009adi}
\bibinfo{author}{P.~Benner}, \bibinfo{author}{R.-C. Li},
  \bibinfo{author}{N.~Truhar},
\newblock \bibinfo{title}{On the {ADI} method for {S}ylvester equations},
\newblock \bibinfo{journal}{J. Comput. Appl. Math.} \bibinfo{volume}{233}
  (\bibinfo{year}{2009}) \bibinfo{pages}{1035--1045}.
\bibitem[{Sylvester(1884)}]{Sylvestermatrix}
\bibinfo{author}{J.~Sylvester},
\newblock \bibinfo{title}{Sur l'equations en matrices px=xq},
\newblock \bibinfo{journal}{C.R. Acad. Sci. {P}aris} \bibinfo{volume}{99}
  (\bibinfo{year}{1884}) \bibinfo{pages}{67--71}.
\bibitem[{Lu and Wachspress(1991)}]{lu1991solution}
\bibinfo{author}{A.~Lu}, \bibinfo{author}{E.~L. Wachspress},
\newblock \bibinfo{title}{Solution of {L}yapunov equations by alternating
  direction implicit iteration},
\newblock \bibinfo{journal}{Comput. \& Math. Appl.} \bibinfo{volume}{21}
  (\bibinfo{year}{1991}) \bibinfo{pages}{43--58}.
\bibitem[{Peaceman and Rachford(1955)}]{peaceman1955numerical}
\bibinfo{author}{D.~W. Peaceman}, \bibinfo{author}{H.~H. Rachford},
\newblock \bibinfo{title}{The numerical solution of parabolic and elliptic
  differential equations},
\newblock \bibinfo{journal}{J. Soc. Ind. Appl. Math.} \bibinfo{volume}{3}
  (\bibinfo{year}{1955}) \bibinfo{pages}{28--41}.
\bibitem[{Golub and Van~Loan(2012)}]{golub2012matrix}
\bibinfo{author}{G.~H. Golub}, \bibinfo{author}{C.~F. Van~Loan},
  \bibinfo{title}{Matrix computations}, volume~\bibinfo{volume}{3},
  \bibinfo{publisher}{JHU Press}, \bibinfo{year}{2012}.
\bibitem[{Akhiezer(1990)}]{akhiezer1990elements}
\bibinfo{author}{N.~I. Akhiezer}, \bibinfo{title}{Elements of the theory of
  elliptic functions}, volume~\bibinfo{volume}{79}, \bibinfo{publisher}{Amer.
  Math. Soc.}, \bibinfo{year}{1990}.
\bibitem[{Starke(1992)}]{starke1992near}
\bibinfo{author}{G.~Starke},
\newblock \bibinfo{title}{Near-circularity for the rational {Z}olotarev problem
  in the complex plane},
\newblock \bibinfo{journal}{J. Approx. Theory} \bibinfo{volume}{70}
  (\bibinfo{year}{1992}) \bibinfo{pages}{115--130}.
\bibitem[{Zolotarev(1877)}]{zolotarev1877application}
\bibinfo{author}{E.~Zolotarev},
\newblock \bibinfo{title}{Application of elliptic functions to questions of
  functions deviating least and most from zero},
\newblock \bibinfo{journal}{Zap. Imp. Akad. Nauk. St. Petersburg}
  \bibinfo{volume}{30} (\bibinfo{year}{1877}) \bibinfo{pages}{1--59}.
\bibitem[{Ganelius(1979)}]{ganelius1979rational}
\bibinfo{author}{T.~Ganelius},
\newblock \bibinfo{title}{Rational functions, capacities and approximation},
\newblock \bibinfo{journal}{Aspects of contemporary complex analysis}
  (\bibinfo{year}{1979}) \bibinfo{pages}{409--414}.
\bibitem[{Gonchar(1969)}]{gonchar1969zolotarev}
\bibinfo{author}{A.~A. Gonchar},
\newblock \bibinfo{title}{{Z}olotarev problems connected with rational
  functions},
\newblock \bibinfo{journal}{Matematicheskii Sbornik} \bibinfo{volume}{120}
  (\bibinfo{year}{1969}) \bibinfo{pages}{640--654}.
\bibitem[{Wachspress(1962)}]{wachspress1962optimum}
\bibinfo{author}{E.~L. Wachspress},
\newblock \bibinfo{title}{Optimum alternating-direction-implicit iteration
  parameters for a model problem},
\newblock \bibinfo{journal}{J. Soc. Ind. Appl. Math.} \bibinfo{volume}{10}
  (\bibinfo{year}{1962}) \bibinfo{pages}{339--350}.
\bibitem[{Smith(1968)}]{smith1968matrix}
\bibinfo{author}{R.~Smith},
\newblock \bibinfo{title}{Matrix equation {XA+BX=C}},
\newblock \bibinfo{journal}{SIAM J. Appl. Math.} \bibinfo{volume}{16}
  (\bibinfo{year}{1968}) \bibinfo{pages}{198--201}.
\bibitem[{Horn and Kittaneh(1998)}]{horn1998two}
\bibinfo{author}{R.~A. Horn}, \bibinfo{author}{F.~Kittaneh},
\newblock \bibinfo{title}{Two applications of a bound on the {H}adamard product
  with a {C}auchy matrix},
\newblock \bibinfo{journal}{Electron. J. Linear Algebra} \bibinfo{volume}{3}
  (\bibinfo{year}{1998}) \bibinfo{pages}{4--12}.
\bibitem[{Braess and Hackbusch(2009)}]{braess2009efficient}
\bibinfo{author}{D.~Braess}, \bibinfo{author}{W.~Hackbusch},
\newblock \bibinfo{title}{On the efficient computation of high-dimensional
  integrals and the approximation by exponential sums},
\newblock in: \bibinfo{booktitle}{Multiscale, nonlinear and adaptive
  approximation}, \bibinfo{publisher}{Springer}, \bibinfo{year}{2009}, pp.
  \bibinfo{pages}{39--74}.
\bibitem[{Kressner(2017)}]{kressnerpersonal}
\bibinfo{author}{D.~Kressner}, \bibinfo{year}{2017}. \bibinfo{note}{Personal
  communication}.
\bibitem[{Olver et~al.(2010)Olver, Lozier, and Boisvert}]{olver2010nist}
\bibinfo{author}{F.~W. Olver}, \bibinfo{author}{D.~W. Lozier},
  \bibinfo{author}{R.~F. Boisvert}, \bibinfo{title}{NIST handbook of
  mathematical functions}, \bibinfo{publisher}{Cambridge University Press},
  \bibinfo{year}{2010}.
\bibitem[{Gugercin et~al.(2003)Gugercin, Sorensen, and
  Antoulas}]{gugercin2003modified}
\bibinfo{author}{S.~Gugercin}, \bibinfo{author}{D.~C. Sorensen},
  \bibinfo{author}{A.~C. Antoulas},
\newblock \bibinfo{title}{A modified low-rank {S}mith method for large-scale
  {L}yapunov equations},
\newblock \bibinfo{journal}{Numer. Algor.} \bibinfo{volume}{32}
  (\bibinfo{year}{2003}) \bibinfo{pages}{27--55}.
\bibitem[{Bebendorf(2008)}]{bebendorf2008hierarchical}
\bibinfo{author}{M.~Bebendorf}, \bibinfo{title}{Hierarchical matrices, volume
  63 of Lecture Notes in Computational Science and Engineering},
  \bibinfo{publisher}{Springer-Verlag, Berlin}, \bibinfo{year}{2008}.
\bibitem[{Badea and Beckermann(2013)}]{badea2013spectral}
\bibinfo{author}{C.~Badea}, \bibinfo{author}{B.~Beckermann},
  \bibinfo{title}{Spectral Sets}, \bibinfo{publisher}{Chapman and Hall/CRC},
  \bibinfo{address}{Chapter 37 of L. Hogben, Handbook of Linear Algebra, second
  edition}, \bibinfo{year}{2013}.
\bibitem[{Shank(2014)}]{shank2014low}
\bibinfo{author}{S.~D. Shank}, \bibinfo{title}{Low-rank Solution Methods for
  Large-scale Linear Matrix Equations}, Ph.D. thesis, Temple University,
  \bibinfo{year}{2014}.
\bibitem[{Beckermann(2011)}]{beckermann2011error}
\bibinfo{author}{B.~Beckermann},
\newblock \bibinfo{title}{An error analysis for rational {G}alerkin projection
  applied to the {S}ylvester equation},
\newblock \bibinfo{journal}{SIAM J. Numer. Anal.} \bibinfo{volume}{49}
  (\bibinfo{year}{2011}) \bibinfo{pages}{2430--2450}.
\bibitem[{Olver and Townsend(2013)}]{olver2013fast}
\bibinfo{author}{S.~Olver}, \bibinfo{author}{A.~Townsend},
\newblock \bibinfo{title}{A fast and well-conditioned spectral method},
\newblock \bibinfo{journal}{SIAM Review} \bibinfo{volume}{55}
  (\bibinfo{year}{2013}) \bibinfo{pages}{462--489}.
\bibitem[{Townsend(2014)}]{townsend2014computing}
\bibinfo{author}{A.~Townsend}, \bibinfo{title}{Computing with functions in two
  dimensions}, Ph.D. thesis, University of Oxford, \bibinfo{year}{2014}.
\bibitem[{Driscoll et~al.(2014)Driscoll, Hale, and Trefethen}]{Chebfun}
\bibinfo{editor}{T.~A. Driscoll}, \bibinfo{editor}{N.~Hale},
  \bibinfo{editor}{L.~N. Trefethen} (Eds.), \bibinfo{title}{Chebfun Guide},
  \bibinfo{publisher}{Pafnuty Publications}, \bibinfo{address}{Oxford},
  \bibinfo{year}{2014}.

\end{thebibliography}

\appendix

\section{On the sharpness of FI-ADI-based error bounds}\label{sec:nearbest} 
Here, we demonstrate that there is a Sylvester matrix equation $AX - XB = F$ satisfying Theorem~\ref{thm: circexp} such that for all $t \leq \rho(\rho+1)/2$, where \smash{$\rank(F) = \rho$}, the rank $ \leq t$ approximant \smash{$\tilde{X}_t$} constructed as in Section~\ref{sec:BoundsF} satisfies \smash{$\|X - \tilde{X}_t\|_2 \approx \sigma_{t+1}(X)$}.  

Recall that in Section~\ref{sec:BoundsF},  an approximant \smash{$\tilde{X}_t$} is constructed in~\eqref{eq:Tnorms} as the sum of rank 1 matrices of the form \smash{$d_{ij}\vec{w}_{ij}\vec{y}_{ij}$}. If the sets \smash{$\{ \vec{w}_{ij}\}$ and $\{ \vec{y}_{ij}\}$}  each consist of mutually orthogonal vectors, then \smash{$\tilde{X}_t$} is not compressible.  We choose $A$, $B$ and $F$ carefully in the following example so that this condition is satisfied. The solution $X$ has a simple structure, and the tightness of the bound can be verified by directly examining its entries.  

Consider the equation \smash{$ AX + XA^T = F$}, where \smash{$A, F \in  \mathbb{R}^{n \times n}$} and \smash{$n = \rho^2$}.  The matrix $A$  is chosen to satisfy \smash{$(A+I)(A-I)^{-1} = Q$}, where $Q$ is the  scaled circulant shift matrix
\[
Q = \frac{1}{q}\begin{bmatrix} 0 & & & 1 \\ 1 \\ & \ddots \\ && 1 &  0  \end{bmatrix}, 
\]
with \smash{$q = \sqrt{2/c+1}$} and $c > 1$.  The matrix $A$ is normal, and its eigenvalues lie on the circle centered at \smash{$z_0 = -(q^2+1)/(q^2-1)$}
with radius \smash{$\eta = 2q/(q^2-1)$.} Applying $\ell$ iterations of  fADI constructs  \smash{$X^{(\ell)}$}.  By Theorem~\ref{thm:equidisk},  the error equation satisfies \smash{$ \|X - X^{(\ell)}\|_2 \leq q^{-2\ell}$}, with optimal ADI shift parameters given as \smash{$\alpha_j = -1, \beta_j=1 $} for all $j$.  From~\eqref{eq:Z} and~\eqref{eq:Y},  \smash{$X^{(\ell)}$} is given by
\begin{equation} 
\label{eq:Exsmiths}
X^{(\ell)} = 2\sum_{j = 0}^{\ell-1} Q^j (A-I)^{-1} F (A+I)^{-T}(Q^T)^{j}, \qquad \ell \geq 1.
\end{equation}
Now choose \smash{$F = [A-I]_{\left(\; : \; , \; 0 : \rho :  n-1\right)} \Lambda \left( [A+I]_{\left(\; : \; , \; 0 : \rho :  \rho (\rho-1) \right)} \right)^T,$} where indexing begins at $0$, and \smash{$\Lambda = {\rm diag}(1, q^{-2}, \ldots, q^{-2(\rho-1)})$}.

\subsection{A closed form solution} 
\label{sec:closedform}
The matrix $X$ is diagonal, and one can verify that its entries are given by
\begin{equation}
\label{eq:entries}
d_{\rho i+j} = \frac{2+2q^{-2}\sum_{s=1}^{\rho-1} q^{-2((n-1)-s(\rho-1))}}{1-q^{-2n}} + 2\sum_{s=1}^{i} q^{-2((\rho-1)(i-s)+j+i)} \approx 2q^{-2(j+i)}
\end{equation}
for \smash{$0\leq i\leq  \rho  -1$} and \smash{$0\leq j\leq \rho-1$}.  The singular values of $X$, plotted in Figure~\ref{fig:nearbest}, are the absolute values of its diagonal entries, sorted in nonincreasing order. 
The rightmost estimate in~\eqref{eq:entries} shows that when \smash{$j+i\leq \rho-1$,} we expect \smash{$j \!+\!i\!+\!1$} entries of \smash{$d$} to be of  magnitude approximately \smash{$2q^{-2(j+i)}$}. This explains the step-like pattern observed in Figure~\ref{fig:nearbest}. 

\begin{figure}
\centering
 \begin{minipage}{.55\textwidth}
 \begin{overpic}[width=\textwidth]{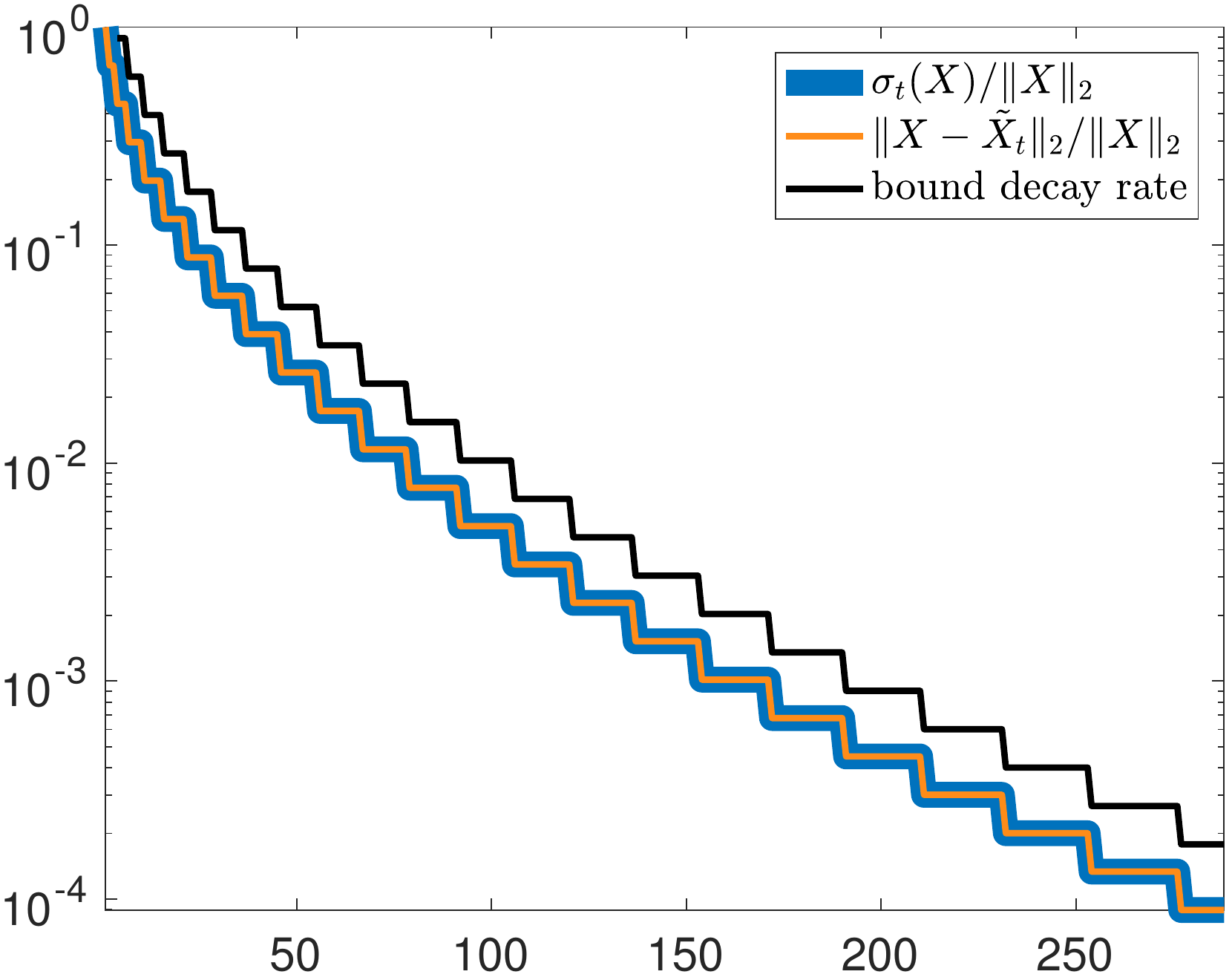}
\put(50,-4) {\small{$t$}} 
\end{overpic}
\end{minipage}
\caption{ The magnitude of normalized singular values for  \smash{$X$} is  plotted against the indices (blue). The error \smash{$\|X - \tilde{X}_t\|/ \|X\|_2$}, where \smash{$\tilde{X}_t$} is the rank $t$ approximant constructed with method used in Section~\ref{sec:Smiths} (orange) differs from the $t+1$ normalized singular value of $X$ only by a small constant factor (indistinguishable to the eye). The decay rate of the bound in Theorem~\ref{thm: circexp} (excluding the polynomial factor) is also shown (black).}
\label{fig:nearbest}
\end{figure}

\subsection{The approximant $\tilde{X}_t$} 
We now consider the approximant \smash{$\tilde{X}_t$} constructed as in Section~\ref{sec:BoundsF}.  For $0 \leq i \leq \rho-1$, let \smash{$X_i$} satisfy
\begin{equation} 
\label{eq:Xi}
AX_i + X_iA^T = \underbrace{\lambda_{i,i} \left[ A - I\right]_{( \; : \; , \; i\rho)}  \left[ A + I\right]_{(\; : \; , \; i\rho)}^T}_{F_i}, \qquad 0 \leq i \leq \rho-1,
\end{equation} 
where \smash{$\lambda_{i,i} = q^{-2i}$} is the $(i, i)$th entry of \smash{$\Lambda$}. Let $t = k(k+1)/2$ with $k \leq \rho$.  For $ i \leq k-1$, construct the approximants \smash{$X_i^{(s_i)}$, $s_i= k  \! - \! i$},  by applying \smash{$s_i$} fADI iterations to~\eqref{eq:Xi}, so that \smash{$\tilde{X}_t = \sum_{i = 0}^{k-1}  X_i^{(s_i)}$}. 
Using~\eqref{eq:Exsmiths} and recalling that $Q$ is a circulant shift matrix, each  \smash{$X_i^{(s_i)}$} is given by 
\[
X_i^{(s_i)} =  {\rm diag}( 0 , \ldots, 0, 2q^{-2i} D_{i}, 0 \ldots, 0),  \qquad X_i^{(s_i)} \in \mathbb{R}^{n \times n},
\]
where \smash{$D_i = {\rm diag}(1,q^{-2}, \ldots,q^{-2(k-i-1)}, 0, \ldots,0)\in\mathbb{R}^{\rho \times \rho }$}, and the first nonzero diagonal entry of \smash{$X_i^{(s_i)}$} is the $(i \rho, i \rho)$ entry of \smash{$D_i$}.  Comparing these matrices to $X$, we observe that the nonzero entries of $D_i$ in \smash{$X_i^{(s_i)}$}  satisfy $(D_i)_{j,j} \approx d_{i \rho+j} $. 

It follows that for $t = k (k +1)/2$, the error \smash{$\mathcal{E}_t : = X - \tilde{X}_t$} has the property that \smash{$\|\mathcal{E}_t\|_2 \approx 2q^{-2k}$}.  Inspecting the entries in $d$ reveals that \smash{$\sigma_{t+1}(X) \approx 2q^{-2k}$}, so that \smash{$\sigma_{t+1}(X)  \approx \|\mathcal{E}_t\|_2$} (see Figure~\ref{fig:nearbest}). As in Theorem~\ref{thm: circexp}, one can also bound non-triangular numbers  (see Section~\ref{sec:BoundsF}). When $t > \rho (\rho +1)/2$, the decay rate of the singular values of $X$ becomes supergeometric, and the approximation error \smash{$\|\mathcal{E}_t\|_2$} does not capture this behavior. 

\end{document}